\documentclass[a4paper,11pt]{amsart}

\usepackage{amsmath,amsxtra,amssymb,latexsym, amscd,amsthm,dsfont,subcaption}
\usepackage[mathscr]{eucal}
\usepackage{mathrsfs,bbm,enumerate,calc,yhmath,mathtools,faktor}
\usepackage{stackengine}
\usepackage{pict2e}
\usepackage{tikz}
\usepackage{graphicx}
\usepackage{color}
\usepackage[foot]{amsaddr}

\usepackage[a4paper]{geometry}
\usepackage{hyperref}
\hypersetup{
colorlinks=true, 
breaklinks=true, 
urlcolor= blue, 
linkcolor= black, 
citecolor=black, 
pdftitle={Moebius inversion}, 
pdfauthor={Nalini Anantharaman and Laura Monk}, 
pdfsubject={} 
}
\usepackage[capitalize]{cleveref}
\crefname{equation}{equation}{equations}
\Crefname{equation}{Equation}{Equations}

\numberwithin{equation}{section}

\title[Moebius inversion]{A Moebius inversion formula \protect\\ to discard tangled hyperbolic surfaces}

\author{Nalini Anantharaman\textsuperscript{1} and Laura Monk\textsuperscript{2}}

\address[1]{Coll\`ege de France, 11 place Marcelin Berthelot, 75005 Paris / IRMA, 7 rue Ren\'e Descartes, 67084 Strasbourg Cedex, France} 

\address[2]{School of Mathematics, University of Bristol, Bristol BS8 1UG, U.K.}

\email{nalini.anantharaman@college-de-france.fr}
\email{laura.monk@bristol.ac.uk}
 
\makeatletter
\@namedef{subjclassname@2020}{\textup{2020} Mathematics Subject Classification}
\makeatother



\theoremstyle{plain}
\newtheorem{thm}{Theorem}[section]
\newtheorem{prp}[thm]{Proposition}
\newtheorem{cor}[thm]{Corollary}
\newtheorem{lem}[thm]{Lemma}

\newtheorem*{namedthm}{\namedthmname}
\newcounter{namedthm}



\theoremstyle{definition}
\newtheorem{defa}[thm]{Definition}

\newtheorem{rem}[thm]{Remark}

\newtheorem{exa}[thm]{Example}
\newtheorem{nota}[thm]{Notation}

\makeatletter
\newcommand*{\ov}[1]{%
  $\m@th\overline{\mbox{#1}}$%
}
\newcommand*{\ovA}[1]{%
  $\m@th\overline{\mbox{#1}\raisebox{3mm}{}}$%
}
\newcommand*{\ovB}[1]{%
  $\m@th\overline{\mbox{#1\rule{0pt}{3mm}}}$%
}
\newcommand*{\ovC}[1]{%
  $\m@th\overline{\mbox{#1\strut}}$%
}
\newcommand*{\ovD}[1]{%
  $\m@th\overline{\mbox{#1\vphantom{\"A}}}$%
}
\newcommand*{\ovE}[1]{%
  $\m@th\overline{\raisebox{0pt}[1.2\height]{#1}}$%
}
\newcommand*{\ovF}[1]{%
  $\m@th\overline{\raisebox{0pt}[\dimexpr\height+0.3mm\relax]{#1}}$%
}
\newcommand*{\ovG}[1]{%
  $\m@th\overline{\raisebox{0pt}[\dimexpr\height+1mm\relax]{#1\vphantom{A}}}$%
}
\makeatother

\newcommand\runderset[2][\sim]{\mathrel{\ensurestackMath{%
  \stackengine{-.2pt}{\scriptscriptstyle#2}{#1}{O}{c}{F}{F}{S}}}}
\newcommand{\N}{\mathbb{N}}
\newcommand{\Z}{\mathbb{Z}}

\newcommand{\R}{\mathbb{R}}

\DeclareMathOperator\argsh{argsh}

\DeclareSymbolFont{extraup}{U}{zavm}{m}{n}
\DeclareMathSymbol{\varheart}{\mathalpha}{extraup}{86}
\DeclareMathSymbol{\vardiamond}{\mathalpha}{extraup}{87}
\newcommand{\nwc}{\newcommand}
\nwc{\mf}{\mathbf} %Latex (as in \bf not tilted math letters)
\nwc{\blds}{\boldsymbol} %Latex 
\nwc{\ml}{\mathcal} %Latex

\nwc{\lam}{\lambda}
\nwc{\del}{\delta}
\nwc{\Del}{\Delta}
\nwc{\Lam}{\Lambda}
\nwc{\elll}{\ell}
\nwc{\IA}{\mathbb{A}} %algebraic
\nwc{\IB}{\mathbb{B}} %ball
\nwc{\IC}{\mathbb{C}} %complex
\nwc{\ID}{\mathbb{D}} %Dedekind
\nwc{\IE}{\mathbb{E}} %Euklides
\nwc{\IF}{\mathbb{F}} %finite field
\nwc{\IG}{\mathbb{G}} %Gauss
\nwc{\IH}{\mathbb{H}} %Hilbert\N-subgroup
\nwc{\IN}{\mathbb{N}} %natural
\nwc{\IP}{\mathbb{P}} %prime
\nwc{\IQ}{\mathbb{Q}} %rational
\nwc{\IR}{\mathbb{R}} %real
\nwc{\IS}{\mathbb{S}} %sphere
\nwc{\IT}{\mathbb{T}} %torus
\nwc{\IZ}{\mathbb{Z}} %integers
\def\bbbone{{\mathchoice {1\mskip-4mu {\rm{l}}} {1\mskip-4mu {\rm{l}}}
{ 1\mskip-4.5mu {\rm{l}}} { 1\mskip-5mu {\rm{l}}}}}
\def\bbleft{{\mathchoice {[\mskip-3mu {[}} {[\mskip-3mu {[}}{[\mskip-4mu {[}}{[\mskip-5mu {[}}}}
\def\bbright{{\mathchoice {]\mskip-3mu {]}} {]\mskip-3mu {]}}{]\mskip-4mu {]}}{]\mskip-5mu {]}}}}
\nwc{\setK}{\bbleft 1,K \bbright}
\nwc{\setN}{\bbleft 1,\cN \bbright}
 
\newcommand{\tfL}{R}

\nwc{\va}{{\bf a}}
\nwc{\vb}{{\bf b}}
\nwc{\vc}{{\bf c}}
\nwc{\vd}{{\bf d}}
\nwc{\ve}{{\bf e}}
\nwc{\vf}{{\bf f}}
\nwc{\vg}{{\bf g}}
\nwc{\vh}{{\bf h}}
\nwc{\vi}{{\bf i}}
\nwc{\vI}{{\bf I}}
\nwc{\vj}{{\bf j}}
\nwc{\vk}{{\bf k}}
\nwc{\vl}{{\bf l}}
\nwc{\vm}{{\bf m}}
\nwc{\vM}{{\bf M}}
\nwc{\vn}{{\bf n}}
\nwc{\vo}{{\it o}}
\nwc{\vp}{{\bf p}}
\nwc{\vq}{{\bf q}}
\nwc{\vr}{{\bf r}}
\nwc{\vs}{{\bf s}}
\nwc{\vt}{{\bf t}}
\nwc{\vu}{{\bf u}}
\nwc{\vv}{{\bf v}}
\nwc{\vw}{{\bf w}}
\nwc{\vx}{{\bf x}}
\nwc{\vy}{{\bf y}}
\nwc{\vz}{{\bf z}}
\nwc{\bal}{\blds{\alpha}}
\nwc{\bep}{\blds{\epsilon}}
\nwc{\barbep}{\overline{\blds{\epsilon}}}
\nwc{\bnu}{\blds{\nu}}
\nwc{\bmu}{\blds{\mu}}
\nwc{\bet}{\blds{\eta}}
\nwc{\rK}{\mathrm{K}}
\nwc{\bD}{\mathbf{D}}
\nwc{\n}{n_{\mathbf{S}}}
\nwc{\g}{g_{\mathbf{S}}}

\nwc{\bk}{\blds{k}}
\nwc{\bm}{\blds{m}}
\nwc{\bM}{\blds{M}}
\nwc{\bp}{\blds{p}}
\nwc{\bq}{\blds{q}}
\nwc{\bn}{\blds{n}}
\nwc{\bv}{\blds{v}}
\nwc{\bw}{\blds{w}}
\nwc{\bx}{\blds{x}}
\nwc{\bxi}{\blds{\xi}}
\nwc{\by}{\blds{y}}
\nwc{\bz}{\blds{z}}

\nwc{\cA}{\ml{A}}
\nwc{\cB}{\ml{B}}
\nwc{\cC}{\ml{C}}
\nwc{\cD}{\ml{D}}
\nwc{\cE}{\ml{E}}
\nwc{\cF}{\ml{F}}
\nwc{\cG}{\ml{G}}
\nwc{\cH}{\ml{H}}
\nwc{\cI}{\ml{I}}
\nwc{\cJ}{\ml{J}}
\nwc{\cK}{\ml{K}}
\nwc{\cL}{\ml{L}}
\nwc{\cM}{\ml{M}}
\nwc{\cN}{\ml{N}}
\nwc{\cO}{\ml{O}}
\nwc{\cP}{\ml{P}}
\nwc{\cQ}{\ml{Q}}
\nwc{\cR}{\ml{R}}
\nwc{\cS}{\ml{S}}
\nwc{\cT}{\ml{T}}
\nwc{\cU}{\ml{U}}
\nwc{\cV}{\ml{V}}
\nwc{\cW}{\ml{W}}
\nwc{\cX}{\ml{X}}
\nwc{\cY}{\ml{Y}}
\nwc{\cZ}{\ml{Z}}

\nwc{\fA}{\mathfrak{a}}
\nwc{\fB}{\mathfrak{b}}
\nwc{\fC}{\mathfrak{c}}
\nwc{\fD}{\mathfrak{d}}
\nwc{\fE}{\mathfrak{e}}
\nwc{\fF}{\mathfrak{f}}
\nwc{\fG}{\mathfrak{g}}
\nwc{\fH}{\mathfrak{h}}
\nwc{\fI}{\mathfrak{i}}
\nwc{\fJ}{\mathfrak{j}}
\nwc{\fK}{\mathfrak{k}}
\nwc{\fL}{\mathfrak{l}}
\nwc{\fM}{\mathfrak{m}}
\nwc{\fN}{\mathfrak{n}}
\nwc{\fO}{\mathfrak{o}}
\nwc{\fP}{\mathfrak{p}}
\nwc{\fQ}{\mathfrak{q}}
\nwc{\fR}{\mathfrak{r}}
\nwc{\fS}{\mathfrak{s}}
\nwc{\fT}{\mathfrak{t}}
\nwc{\fU}{\mathfrak{u}}
\nwc{\fV}{\mathfrak{v}}
\nwc{\fW}{\mathfrak{w}}
\nwc{\fX}{\mathfrak{x}}
\nwc{\fY}{\mathfrak{y}}
\nwc{\fZ}{\mathfrak{z}}

\nwc{\bT}{\mathbf{T}}

\nwc{\tA}{\widetilde{A}}
\nwc{\tB}{\widetilde{B}}
\nwc{\tE}{E^{\vareps}}
\nwc{\tk}{\tilde k}
\nwc{\tN}{\tilde N}
\nwc{\tP}{\widetilde{P}}
\nwc{\tQ}{\widetilde{Q}}
\nwc{\tR}{\widetilde{R}}
\nwc{\tV}{\widetilde{V}}
\nwc{\tW}{\widetilde{W}}
\nwc{\ty}{\tilde y}
\nwc{\teta}{\tilde \eta}
\nwc{\tdelta}{\tilde \delta}
\nwc{\tlambda}{\tilde \lambda}
\nwc{\ttheta}{\tilde \theta}
\nwc{\tvartheta}{\tilde \vartheta}
\nwc{\tPhi}{\widetilde \Phi}
\nwc{\tpsi}{\tilde \psi}
\nwc{\tmu}{\tilde \mu}

\nwc{\To}{\rightarrow} %limits

\nwc{\ad}{\rm ad}
\nwc{\eps}{\epsilon}
\nwc{\ep}{\epsilon}
\nwc{\vareps}{\varepsilon}

\def\ep{\epsilon}

\def\sq2{\sqrt{2}}

\def\t2{{\mathbb T}^2}
\def\s2{{\mathbb S}^2}

\def\N{\mathbb{N}}

\def\R{\mathbb{R}}

\def\Z{\mathbb{Z}}

\def\O{\mathcal{O}}

\nwc{\lap}{\bigtriangleup}
\nwc{\rest}{\restriction}
\nwc{\Diff}{\operatorname{Diff}}
\nwc{\diam}{\operatorname{diam}}
\nwc{\Res}{\operatorname{Res}}
\nwc{\Spec}{\operatorname{Spec}}
\nwc{\Vol}{\operatorname{Vol}}
\nwc{\Op}{\operatorname{Op}}
\nwc{\supp}{\operatorname{supp}}
\nwc{\Span}{\operatorname{span}}

\nwc{\dia}{\varepsilon}
\nwc{\cut}{f}
\nwc{\qm}{u_\hbar}

\def\hto0{\xrightarrow{\hbar\to 0}}

\def\rto0{\xrightarrow{r\to 0}}

\nwc{\la}{\langle}
\nwc{\ra}{\rangle}
\nwc{\lp}{\left(}
\nwc{\rp}{\right)}

\nwc{\bequ}{\begin{equation}}
\nwc{\be}{\begin{equation}}
\nwc{\ben}{\begin{equation*}}
\nwc{\bea}{\begin{eqnarray}}
\nwc{\bean}{\begin{eqnarray*}}
\nwc{\bit}{\begin{itemize}}
\nwc{\bver}{\begin{verbatim}}

\nwc{\eequ}{\end{equation}}
\nwc{\ee}{\end{equation}}
\nwc{\een}{\end{equation*}}
\nwc{\eea}{\end{eqnarray}}
\nwc{\eean}{\end{eqnarray*}}
\nwc{\eit}{\end{itemize}}
\nwc{\ever}{\end{verbatim}}

\makeatletter
\newlength{\temp@wc@width}
\newlength{\temp@wc@height}
\newcommand{\widecheck}[1]{%
  \setlength{\temp@wc@width}{\widthof{$#1$}}%
  \setlength{\temp@wc@height}{\heightof{$#1$}}%
  #1\hspace{-\temp@wc@width}%
  \raisebox{\temp@wc@height+2pt}[\heightof{$\widehat{#1}$}]%
     {\rotatebox[origin=c]{180}{\vbox to 0pt{\hbox{$\widehat{\hphantom{#1}}$}}}}%
}
\makeatother

\newcommand{\Pwpo}{\mathbb{P}_g^{\mathrm{\scriptsize{WP}}}}

\newcommand{\Pwp}[1]{\Pwpo \left( #1 \right)}

\newcommand{\tfP}{\mathrm{TF}_{\cP}}

\renewcommand{\O}[2][ ]{\mathcal{O}_{#1} \left( #2 \right)}

\newcommand{\eqc}[1]{[ #1 ]_{\mathrm{loc}}}
\newcommand{\eq}{\, \raisebox{-1mm}{$\runderset{\mathrm{loc}}$} \,}

\makeatletter
\newcommand{\smallbullet}{} % for safety
\DeclareRobustCommand\smallbullet{%
  \mathord{\mathpalette\smallbullet@{0.5}}%
}
\newcommand{\smallbullet@}[2]{%
  \, \vcenter{\hbox{\scalebox{#2}{$\m@th#1\bullet$}}} \,%
}
\makeatother

\newcommand{\inj}{\kappa}

\newcommand{\tf}[1][g]{\mathrm{TF}^{\inj,\tfL}_{#1}}
\newcommand{\ft}{\mathbf{S}}
\newcommand{\ct}{\mathbf{c}}

\newcommand{\restr}{|}
\newcommand{\hbis}{\tilde{h}}

\begin{document}

\begin{abstract}
  Recent literature on Weil--Petersson random hyperbolic surfaces has met a
  consistent obstacle: the necessity to condition the model, prohibiting certain rare geometric
  patterns (which we call \emph{tangles}), such as short closed geodesics or embedded surfaces of
  short boundary length. The main result of this article is a Moebius inversion formula, allowing to
  integrate the indicator function of the set of tangle-free surfaces in a systematic, tractable
  way. It is inspired by a key step of Friedman's celebrated proof of Alon's conjecture. We further
  prove that our tangle-free hypothesis significantly reduces the number of local topological types
  of short geodesics, replacing the exponential proliferation observed on tangled surfaces by a
  polynomial growth.
\end{abstract}

\maketitle

\section{Introduction}

The study of random compact hyperbolic surfaces of large genus has drawn a lot of attention recently
(see for instance \cite{lipnowski2021, mirzakhani2013, Naud-det, rudnick2022, wu2022a, wu2022} as a
non-exhaustive list of articles). Throughout these works, it has become apparent that there exists a
set of ``bad'' hyperbolic surfaces, of small but non-zero probability, that one could benefit from
discarding -- in particular when studying the spectrum of the laplacian of random surfaces. These
``bad'' hyperbolic surfaces contain geometric patterns called \emph{tangles}, a notion that was
formalized and studied by the second author and Thomas in the context of hyperbolic geometry
\cite{monk2021a}, and first arose in the literature about random graphs. Tangles are responsible for
an exceptional proliferation of closed geodesics, as observed by Lipnowski--Wright in
\cite{lipnowski2021}. In our previous work \cite{Ours1}, we have demonstrated how surfaces with
tangles have a small spectral gap. In this article, we shall provide a systematic way to discard
surfaces containing tangles from the moduli space, without any tedious topological enumeration: this
is what we call the \emph{Moebius inversion formula}. Importantly, we  prove that discarding
tangled surfaces does allow to solve issues related to exponential proliferation of closed
geodesics.

\subsection{History of the word tangle in the graph setting}

The word \emph{tangle} was first coined by Friedman his proof of Alon's conjecture related to the
spectral gap of large random regular graphs \cite{friedman2003}. The word was later used by
Bordenave with a different meaning, but with the same purpose of discarding a set of small
probability of ``bad graphs'' \cite{bordenave2020}.  In the recent work of Huang and Yau on the fine
spectral statistics of random regular graphs, it is again necessary to condition on a subset of the
full set of regular graphs (the subset $\bar \Omega$ of \cite[Definition 1.1]{HuangYau}),
characterized by the absence of too many cycles in balls of a certain radius.

\subsection{Tangles in hyperbolic surfaces}

In this paper we use the notion of tangles introduced in \cite{monk2021a} for hyperbolic surfaces,
which is {closest} to Bordenave's in the graph setting.  Given two positive real numbers
$\inj$ and $\tfL$, a \emph{$(\inj, \tfL)$-tangle} (or in short, a tangle), is either a simple smooth
closed curve of length less than $\inj$, or a hyperbolic pair of pants or once-holed torus with all
boundary components of length less than $\tfL$.

A hyperbolic surface is said to be \emph{$(\inj, \tfL)$-tangle-free} (or, in short, tangle-free) if
it does not contain any $(\inj, \tfL)$-tangle. We denote by $\tf[g,n]$ the set of
$(\inj, \tfL)$-tangle-free orientable hyperbolic surfaces of signature $(g, n)$. It can be seen as a
subset of the Teichm\"uller space $\cT_{g,n}$ or of the moduli space $\cM_{g,n}$. We write $\tf$,
$\cT_g$, $\cM_g$ when $n=0$, i.e. for compact hyperbolic surfaces.

In several models of random compact hyperbolic surfaces, and for the right choice of parameters
$\inj, \tfL$, the set $\tf$ has probability close to $1$. For instance, for the Weil--Petersson
probability measure $\Pwpo$ on the moduli space $\cM_g$ of compact hyperbolic surfaces of genus $g$,
it is known that for any $\inj \in (0,1)$, for $\tfL= \alpha \log g$ with $\alpha \in (0,2/3)$,
\begin{equation*}
  1 - \Pwp{\tf} = \O{ \inj^2 + g^{\frac 3 2 \alpha-1}}.
\end{equation*}
See \cite[Theorem 4.2]{mirzakhani2013} and \cite[Theorem 5]{monk2021a}. % An application of this was
  % already used in \cite{Ours1}: we could condition on $\bbbone_{\tf}$, losing only a fraction
  % $\inj^2 + \O{g^{\frac 3 2 \inj-1}}$ of the total moduli space, and avoid the exponential
  % proliferation of topologies of geodesics of length comparable to $\log g$.

\subsection{Polynomial counting result}

Our first main result is the following, showing that, for $L>0$, the number of topologies of
closed geodesics of length $\leq L$ in a tangle-free hyperbolic surface is polynomial in $L$
(instead of the expected exponential growth), provided the ratio $L/\tfL$ is bounded.

 \begin{thm}
  \label{t:TF_curves-intro}
  Let $\mathrm{Loc}_{\chi}^{\inj,\tfL, L}$ be the {set} of local topological types of closed
  geodesics of length $\leq L$ that can arise in a $(\inj, \tfL)$-tangle-free surface of Euler
  characteristic $-\chi$.  Then, \begin{equation*} \# \mathrm{Loc}_{\chi}^{\inj,\tfL, L} = \O[
    \chi]{ \frac L{\tfL} \Big(1+\frac{L}{\inj}\Big)^{18(1+\chi) \frac L{\tfL}} },
  \end{equation*}
 where $\O[ \chi]{\cdot}$ is the usual Landau ``big-O'' notation, with an implied constant depending on $\chi$.
  \end{thm}
  In other words, discarding tangle-free surfaces avoids the exponential proliferation of {local}
  topological types of closed geodesics, up to scales $L$ growing like a multiple of $R$. Precise
  definitions of the notion of \emph{local topological type} and of the set
  $\mathrm{Loc}_{\chi}^{\inj,\tfL, L}$ are given in \S \ref{s:ltt}.

This result is a vast generalization of estimates obtained in \cite{Ours1}, corresponding to the case
$\chi=1$, to arbitrary Euler characteristic.  The theorem bears similarities with a counting result
obtained by Lipnowski and Wright in \cite[Theorem 2.1]{lipnowski2021}. However, there is an
important difference here: Lipnowski and Wright count closed geodesics in a \emph{given} tangle-free
hyperbolic surface (their method of proof refering explicitly to the fixed hyperbolic metric),
whilst our statement applies simultanously to \emph{all} tangle-free hyperbolic surfaces.

\subsection{Conditioning on the set of tangle-free surfaces}

There are two main reasons why conditioning a probability measure can be a dangerous operation.  On
the one hand, whilst the Weil--Petersson measure is a smooth measure, the conditioning introduces an
indicator function with a discontinuity. This is important when analytic tools are used to study
random hyperbolic surfaces, such as integration by parts, as done in \cite{Ours1,rudnick2022}.  On
the other hand, in the special case of the Weil--Petersson probability measure $\Pwpo$, multiplying
the measure by the indicator function $\bbbone_{\tf[g,n]}$ completely prevents the application of
one of the main computational tools of the field: the algebraic integration formulas provided by
Mirzakhani's works \cite{mirzakhani2007}.

One might think about the generalized McShane identities proven by Mirzakhani in
\cite{mirzakhani2007}: these can be seen as clever rewritings of the constant function
$\bbbone_{\cM_{g,n}}$ as a sum of ``manageable functions'', i.e. functions which can be unfolded to
the Teichm\"uller space, leading Mirzakhani to her celebrated recursion formulas for the total volume of
$\cM_{g,n}$. The same way that one needs to decompose $\bbbone_{\cM_{g,n}}$ in manageable functions
to compute the total mass of $\cM_{g,n}$, one needs to decompose $\bbbone_{\tf[g,n]}$ as a sum of
manageable functions to condition the Weil--Petersson model by the tangle-free hypothesis.  The
second main result of this paper, which we shall now present, precisely achieves this goal, in the
form of an identity apparented to the family of Moebius inversion formulas.

\subsection{Moebius inversion formulas in number theory and for hyperbolic surfaces}
Let us recall a fundamental construction in number theory, providing an analogy with our forthcoming
construction. Let $\cP$ be a subset of the set of prime numbers, playing the role of
``tangles''. ``Tangle-free'' numbers will be those not divisible by an element of $\cP$: we denote
this set by $\tfP$.  Let us also denote by $\bD_{\cP}$ the set of numbers having no prime factors
outside $\cP$.  We can easily construct by induction a function {$\tilde \mu_{\mathcal{P}}$}
defined on $\IN$, such that for all $n>1$,
% \begin{equation} \label{e:tf_nt}
%   \begin{split}
%     0&= \sum_{\substack{d | n\\   d>1 }} \tilde\mu_{\cP}(d) \mbox{ if } n \in \tfP\\
%     1&= \sum_{\substack{d | n\\   d>1}} \tilde\mu_{\cP}(d) \mbox{ if } n \not\in \tfP.
%   \end{split}
% \end{equation}
\begin{equation} \label{e:tf_nt}
  \sum_{\substack{d | n\\   d>1 }} \tilde\mu_{\cP}(d)
  =
  \begin{cases}
    0 & \text{if } n \in \tfP\\
    1 & \text{otherwise.}
  \end{cases}
\end{equation}
%\laur{j'ai changé la présentation pour rendre l'égalité plus lisible, mais il y a des bons côtés
%  l'autre version, si tu la préfères tu peux la restaurer, elle est en commentaire ci-dessus}

The function $\tilde\mu_{\cP}$ can be seen to be uniquely defined by induction, and admits the explicit expression: $\tilde\mu_{\cP}(n)=(-1)^{j+1}$ if $n$ is square-free and can be written
as the product of $j$ distinct prime numbers belonging to $\cP$,  $\tilde\mu_{\cP}(n)=0$ otherwise. In particular, this implies that $\tilde\mu_{\cP}(n)=0$ if $n\not\in \bD_{\cP}$.
Note that \eqref{e:tf_nt} may be interpreted as a
rewriting of the indicator function of $\tfP$ as a sum over all divisors of an integer, or,
alternatively, as a form of inclusion-exclusion:
\begin{align*}  
 1-  \bbbone_{\tfP}(n)= \sum_{\substack{d | n\\ d>1} } \tilde\mu_{\cP}(d),
\end{align*}
equivalently
\begin{align*}  
   \bbbone_{\tfP}(n)=1- \sum_{\substack{d | n\\ d>1} } \tilde\mu_{\cP}(d)=- \sum_{d | n} \tilde\mu_{\cP}(d)
\end{align*}
where in the last equality we have can extended $\tilde\mu_{\cP}$ by setting $\tilde \mu_{\cP}(1)=-1$.
In other words, the function $\tilde\mu_{\cP}$ can be used to sieve out integers divisible by $\cP$.

If $\cP$ is the full set of prime numbers, then $\tfP$ is reduced to $\{1\}$ and $\tilde\mu_{\cP}$ is the opposite of the standard Moebius
function, which satisfies
for any positive integer $n$,
\begin{equation*}
  \bbbone_{\{1\}}(n) = \sum_{d | n} \mu(d).
\end{equation*}
  The Moebius function $\mu$ satisfies the
multiplicative property $\mu(mn) = \mu(m) \mu(n)$ as soon as $m$ and $n$ are coprime. The function $\tilde\mu_{\cP}$ satisfies $\tilde\mu_{\cP}(mn) = -\tilde\mu_{\cP}(m) \tilde\mu_{\cP}(n)$.

In his seminal work about random regular graphs, J. Friedman uses what he calls a \emph{generalized
  Moebius inversion} to write the indicator function of the set of tangled graphs: see
\cite[Proposition 9.5]{friedman2003}. This is an essential component to his proof of Alon's
conjecture. We shall follow his original idea to sieve out ``tangled'' surfaces (it is because we followed Friedman's construction that our Moebius functions have
a minus sign compared with the ``standard'' case, in particular the multiplicative property comes with an extra sign change).

The first ``Moebius inversion formula'' appearing in the context of random hyperbolic surfaces are
the inclusion-exclusion formulas used by Mirzakhani \cite{mirzakhani2013} and Lipnowski--Wright
\cite{lipnowski2021} to study one type of tangles on random hyperbolic surfaces: short closed
geodesics. For $\inj \in (0,1)$, they write the indicator of the $\inj$-thin set $\cM_g^{\leq \inj}$
(the set of elements of $\cM_g$ of injectivity radius smaller or equal to $\inj$) as
\begin{align}\label{e:usualIE}\bbbone_{\cM_g^{\leq \inj}}(X)
  = \sum_{q=1}^{+\infty} \mu_q
  \sum_{(c_1, \ldots, c_q)\subset X}
  \prod_{i=1}^q \bbbone_{[0, \inj]} (\ell_X(c_i)),\end{align}
where $\mu_q = (-1)^{q+1} / (2^qq!)$, the sum runs over all oriented multicurves $(c_1, \ldots, c_q)\subset
X$ and $\ell_X(c_i)$ denotes the length of the closed geodesic $c_i$ on $X$. 
The strength of this expression is that its right-hand side can be easily integrated on the moduli space,
because one can explicitly describe how to lift it to an integrable function on the Teichm\"uller
space.

In our previous work \cite{Ours1}, we used a similar approach to ``sieve out'' both kinds of
tangles from surfaces: short closed geodesics as well as embedded pairs of pants or
once-holed tori with a short boundary. Indeed, we can write $1- \bbbone_{\tf[g]}$ under the
form~\eqref{e:usualIE}, where the components $(c_1, \ldots, c_{{q}})$ may now either be short curves or
pairs of pants or once-holed tori with a short boundary. However, whilst closed geodesics of length
at most ${\kappa} < 1$ on a hyperbolic surface are disjoint (by the collar lemma), there is no reason why
pairs of pants or once-holed tori with a boundary shorter than $R = \alpha \log g$ would be. This
causes significant challenges in \cite{Ours1}, causing us to compute a painfully explicit
approximation of $1- \bbbone_{\tf[g]}$ up to errors decaying in $1/g^2$. The enumeration is
cumbersome and unsustainable if one wants to reach higher levels of precision.

\subsection{Our Moebius inversion formula}

We here replace these inclusion-exclusion formulas in a systematic manner, allowing for a more
abstract but more easily handled description of the indicator function $1- \bbbone_{\tf[g, n]}$ in
terms of functions which can be explicitly lifted to Teichm\"uller space and integrated.

We define a \emph{big moduli space} ${\mathbf{M}}$ which is, roughly, the set of all compact
hyperbolic surfaces, with all possible topologies (connected or not, with or without boundary --
possibly with some degenerate components reduced to 1-dimensional compact manifolds). The set
${\mathbf{D}}_{\inj, \tfL}\subset {\mathbf{M}}$ is the set of surfaces weakly filled by tangles (see \S
\ref{s:tangles} for detailed definitions). We then prove the following.
\begin{thm}[Theorem \ref{t:moebius}] \label{t:moebius-intro} Given any $\inj, \tfL >0$, there exists a unique function
  $\mu = \mu_{\inj, \tfL} : {\mathbf{M}} \To \R$, such that, for $Z \in \mathbf{M}$,
 \begin{itemize}
 \item if  $Z\notin  {\mathbf{D}}_{\inj, \tfL} $, then $\mu(Z)=0$;
 \item if $Z\in {\mathbf{D}}_{\inj, \tfL} $, then
\begin{align*}
1= \sum_{\tau \in \cS(Z)} \mu(\tau),
\end{align*}
where the sum runs over all possible sub-surfaces of $Z$ with geodesic boundary.
\end{itemize}
\end{thm}

{
We then provide more precise information on the Moebius function: its explicit values in the case of
one-dimensional tangles (Proposition \ref{prp:circles}) as well as upper bounds (Theorem
\ref{thm:boundmu}). Like the original Moebius function in number theory, the function $\mu$
possesses a multiplicative property (Proposition \ref{prp:multi} or Remark \ref{r:multi}). However
our function $\mu$ is less explicit as in the previously cited examples.}

Theorem \ref{t:moebius-intro} has the following corollary, which is the main goal of the whole construction:
 \begin{cor} For every $g$ and $n$, for every hyperbolic surface $X$ of signature $(g, n)$,
\begin{align} \label{e:inversionformula}1- \bbbone_{\tf[g,n]}(X)= \sum_{\tau\in \cS(X)} \mu(\tau).
\end{align}
 \end{cor}

We present the construction of the Moebius function in the special context of getting rid of tangles -- with the purpose of applying it in the study of the spectral gap of random hyperbolic surfaces -- but we believe that the construction is of wider interest, as it can be implemented to remove sets of surfaces containing other kinds of geometric patterns.

\subsection*{Acknowledgements} The existence of such a construction for random graphs was explained to us by J. Friedman, who used it in \cite{friedman2003} to prove optimal spectral gap for random regular graphs; we are deeply thankful to him for sharing his ideas.

This research has received funding from the EPSRC grant EP/W007010/1 and from the European Research Council (ERC) under the
European Union’s Horizon 2020
research and innovation programme (Grant agreement No. 101096550).

\subsection*{Data statement}
No new data were generated or analysed during this study.

\section{Tangles in hyperbolic surfaces}\label{s:tangles}

This section introduces notations and definitions related to surfaces and curves which we use
thoughout this article. Importantly, we define the notion of tangle, derived tangle and tangle-free
surface.

 \subsection{The big moduli space $\mathbf{M}$}  
 The tangles we study in this article can be either 1-dimensional (short closed curves), or
 2-dimensional (pairs of pants or once-holed tori with short boundary). The Moebius function $\mu$
 will be supported in the set of surfaces filled by tangles; these surfaces may be disconnected, and
 may have 1-dimensional connected components. We need an extended notion of moduli space which
 encompasses all these possibilities.

\subsubsection{c-surfaces}

 Let $\N_0 = \{0, 1, \ldots \}$ denote the set of nonnegative integers.  We order the pairs
 $(g, n)\in \N_0^2$ by the order relation: $(g, n)\leq (g', n')$ if $2g+n-2 < 2g'+n'-2$, or if
 $2g+n-2 = 2g'+n'-2$ and $g\leq g'$.

\begin{nota}
  For $q\geq 1$, we denote by ${\mathfrak{S}}^{q}$ the set of $2q$-tuples
$$(\vg, \vn)=((g_1, n_1), \ldots, (g_q, n_q))\in (\N_0^2 {\setminus \{(0,0),(0,1),(1,0)\}})^q$$
satisfying $(g_i, n_i)\leq (g_{i+1}, n_{i+1})$ for $1 \leq i < q$.
\end{nota}

\begin{defa} A \emph{c-surface} of \emph{signature} $(\vg, \vn)\in {\mathfrak{S}}^{q}$ is a
  non-empty topological space $S$ with $q$ connected components $S= \bigsqcup_{j=1}^q \tau_j$
  labelled from $1$ to $q$, where
\begin{itemize}
\item $\tau_i$ is a $1$-dimensional oriented manifold diffeomorphic to a circle if $(g_i, n_i)=(0, 2)$;
\item if $2-2g_i-n_i<0$, $\tau_i$ is a $2$-dimensional orientable manifold with (or without)
  boundary, of signature $(g_i, n_i)$, with a numbering of boundary components by the integers
  $1, \ldots, n_i$.
\end{itemize}
We denote $\chi(\tau_i)= 2g_i-2+n_i$ and call it the \emph{absolute Euler characteristic} of
$\tau_i$.
\end{defa}

The c-surface $S$ can be decomposed as $S=(c, \sigma)$, where
$c=(\tau_1, \ldots, \tau_{q_1})=(c_1, \ldots, c_{q_1})$ is a collection of $1$-dimensional compact
manifolds, and $\sigma=(\tau_{q_1+1}, \ldots, \tau_{q_1+q_2})=(\sigma_1, \ldots, \sigma_{q_2})$ is a
collection of $2$-dimensional orientable compact manifolds with boundary ($q=q_1+q_2$ is the total
number of connected components). We will call $c$ the \emph{$1$-dimensional (or 1d) part of $S$} and
$\sigma$ the \emph{$2$-dimensional (or 2d) part of $S$}.  We say that $S$ is purely 1d if $q_2=0$, and
purely 2d if $q_1=0$. Each individual element $\tau_i$ will be refered to as a
  \emph{component} of $S$; a 1d-component if $i \leq q_1$ and a 2d-component otherwise.  

\begin{nota}
  Let $\chi(S)$ be the total absolute Euler characteristic of $S$, that is
  $$\chi(S)=\sum_{i=1}^q \chi(\tau_i) = \sum_{i=1}^{q_2} \chi(\sigma_i)$$ where
  $\chi(\sigma_i)=\chi(\tau_{q_1+i})=2g_{q_1+i}+n_{q_1+i}-2$.  We denote as $\mathfrak{c}(S)=q_1$
  the number of 1d components of the c-surface $S$, $\mathfrak{s}(S) = q_2$ its number of 2d
  components, and $\mathfrak{q}(S) = q = \mathfrak{c}(S) + \mathfrak{s}(S)$ its total number of
  connected components.
\end{nota}

{
  \begin{nota}
  The boundary $\partial S$ of $S = (\tau_1, \ldots, \tau_q)$ is defined as
  $\partial S = \bigsqcup_{i=1}^{q} \partial \tau_i$, with the convention that the boundary of a
  1d-component $c_i$ is itself (so that the 1d-components are all included in $\partial S$).
\end{nota}

Note that the boundary of $S$ is an ordered family of
$\mathfrak{c}(S)+\sum_{i=\mathfrak{c}(S)+1}^{\mathfrak{q}(S)} n_i$ closed circles. We can view the
boundary as oriented by taking the orientation of the 1d-components, and orienting the boundary of
the 2d components so that the surface lies of their left.}

\subsubsection{Extended notion of moduli space} \label{s:bigM} For $2-2g-n<0$, we define
$\mathcal{T}_{g,n}^*$ to be the set of marked hyperbolic surfaces of signature $(g,n)$, with $n$
labelled boundary geodesics of variable lengths $(x_1, \ldots, x_n) \in \R_{>0}^n$, up to
isotopy. Note that we do not allow for the elements of $\mathcal{T}_{g,n}^*$ to have cusps.

We denote by $\mathcal{M}_{g,n}^*$ the corresponding moduli spaces, obtained by quotienting the
Teichm\"uller space by the action of the Mapping Class Group. Recall that the latter is (modulo
isotopy) the group of orientation preserving homeomorphisms that fix the boundary, so that an
element of $\mathcal{M}_{g,n}^*$ still has a labelled boundary.

 We extend these notations to $(g,n)=(0,2)$, i.e. when $2-2g-n=0$:
 $\mathcal{T}_{0,2}^*= \mathcal{M}_{0,2}^*$ is the space of Riemannian metrics on an oriented
 circle, modulo orientation preserving isometry.
 
 \begin{nota}
   If $Z \in \mathcal{M}_{g,n}^*$, denote by $\ell(\partial Z)$ the total boundary length of $Z$,
   and $\ell^{\mathrm{max}}(\partial Z)$ the length of the longest boundary component of $Z$.
 \end{nota}
 %: as a  $\mathcal{M}_{g,n}^*$ the boundary components are still labelled by $\{1, \ldots, n\}$.

 For $(g,n)=(0,2)$, we also (naturally) write $\ell(Z) = \ell(\partial Z)$, which is the length of
 $Z$. The map $Z\mapsto \ell(Z)$ allows us to identify $\mathcal{M}_{0,2}^*$ with $\R_{>0}$.

 \begin{defa} \label{d:bigM}
    For $(\vg, \vn)=((g_1, n_1), \ldots, (g_q, n_q))\in {\mathfrak{S}}^{q}$, we
    define
    $$\mathcal{M}_{\vg, \vn}^*=\prod_{i=1}^q \mathcal{M}_{g_i,n_i}^*
    \qquad \text{and} \qquad
    \mathcal{T}_{\vg, \vn}^*=\prod_{i=1}^q \mathcal{T}_{g_i,n_i}^*.$$
    We then define the \emph{big moduli space},
 $$\mathbf{M}= \bigsqcup_{q=1}^{+\infty}\bigsqcup_{(\vg, \vn)\in  {\mathfrak{S}}^{q}}\mathcal{M}_{\vg, \vn}^*.$$
 \end{defa}
 For $(\vg,\vn) \in \mathfrak{S}^q$, the space $\cM_{\vg,\vn}^*$ can be interpreted as the moduli
 space of Riemannian metrics on a c-surface with $q$ labelled connected components, of respective
 signatures $(g_1, n_1), \ldots, (g_q, n_q)$, with a hyperbolic metric on the 2d components. The
 space $\mathbf{M}$ encompasses all of these possibilities.

 \subsection{Sub-c-surfaces and their geodesic representatives}
 We define the following notions of sub-elements for our c-surfaces.

 \begin{defa}\label{d:subsurface}
   Let $S$ be a c-surface.
   \begin{enumerate}
   \item We say a closed path on $S$ is \emph{simple} if it has no self-intersection. Non-contractible
      simple closed paths are called \emph{curves}.
   \item A \emph{multicurve} in $S$ is a family of curves $c=(c_1, \ldots, c_j)$ on $S$ such that,
     for all $i \neq i'$, $c_i$ and $c_{i'}$ are disjoint, and $c_{i}$ is homotopic to neither
     $c_{i'}$ nor $c_{i'}^{-1}$.
   \item We call {\em{sub-c-surface}} of $S$ a c-surface $S'$ together with an injection $\iota : S'\To S$
such that $\iota(\partial S')$ is a multicurve in $S$, and such that $\iota$ preserves the orientation of the 2d components of $S'$.
   \end{enumerate}
 \end{defa}

\begin{rem}
By identifying $\iota(S')$ and $S'$, we will actually see a sub-c-surface of $S$ as a subset of $S$. Note, however, that it comes with additional structure: a numbering of the connected components
of $\iota(S')$, a numbering of the boundary components of the 2d-part, and an orientation of the 1d-part. For us, the orientation of the 1d-part plays exactly the same role as the numbering of the boundary components: we think of them as allowing to distinguish between the ``sides'' of connected components of $S'$.  
 \end{rem}

Assume that $S$ has signature $(\vg, \vn)\in {\mathfrak{S}}^{k}$.  Let $S'$ be a sub-c-surface of $S$, and let $S'=(c', \sigma')$ be the
decomposition of $S'$ into its 1d and its 2d parts.  If $S$ is endowed with a metric
$ Z\in \cT^*_{\vg, \vn}$, then the collection $(c', \sigma')$ is isotopic to a {\em{geodesic
    sub-c-surface}} $(c'_Z, \sigma'_Z)$. {By this, we mean that the components of $c'_{Z}$
  together with the boundary components of $\sigma'_{Z}$ form a multigeodesic (a geodesic
  multicurve) on the hyperbolic surface $Z$.}  We call $(c'_Z, \sigma'_Z)$ the {\em{geodesic
      representative}} of $(c', \sigma')$ in $Z$.

 The geodesic representative $(c'_Z, \sigma'_Z)$, endowed with the Riemannian metric induced by $Z$,
 is a representative of an element of the moduli space $\cM^*_{\vg', \vn'}$, where $(\vg',\vn')$ is
 the signature of $S'$.

 \begin{nota}\label{n:sZ}
   Let $Z \in \mathbf{M}$ be a hyperbolic surface.  We denote by $\cS(Z)$ the set of geodesic
   sub-c-surfaces of $Z$. If $\tau\in \cS(Z)$, we denote by $\bar\tau$ its equivalence class in
   $\mathbf{M}$ (i.e. its equivalence class up to isometry, once we forget the marking).
 \end{nota}

\subsection{Subset filling a c-surface}

The following definition is a straightforward adaptation of the usual definition for surfaces.

{
\begin{defa}\label{d:fill}
  Let $S$ be a c-surface and $E$ be a subset of $S$. 
  % For all $\eps>0$, denote by $\cR_\eps(E)$ the $\eps$-neighbourhood of $E$.
  \begin{enumerate}
  \item We say that $E$ \emph{fills} $S$ if the connected components of $S\setminus E$ are either
    contractible, or annular regions (i.e. cylinders) around a boundary component of the 2d-part of $S$.
  \item We say that $E$ \emph{weakly fills} $S$ if the connected components of $S\setminus E$ are
    either contractible, or annular regions in the 2d-part of $S$.
  \end{enumerate}
\end{defa}

In the following, we will only consider subsets $E$ which can be written as
$E = \bigcup_{t \in I}S_t$, where $S_t$ are sub-c-surfaces of $S$ (the union being not necessarily
disjoint). In this case,
we notice that a (weakly) filling set $E$ must contain all the 1d-components of $S$.

\begin{defa}
  Let $S$ be a c-surface and $E$ be a subset of $S$. We call \emph{c-surface weakly filled by $E$}
  the minimal sub-c-surface of $S$ containing $E$ (in the sense of inclusion), up to isotopy.
\end{defa}

The c-surface weakly filled by $E$ is, as the naming suggests, weakly filled by $E$. 
}

 \subsection{Definition of tangles}
 We introduce a notion of \emph{tangles} and \emph{derived tangles}, associated with two parameters $\inj, \tfL >0$ with $\inj< \tfL$. 

\begin{defa}
  \label{def:tangle}
  Let $Z \in \mathbf{M}$.
  \begin{enumerate}
  \item We say $Z$ is a {\em{$(\inj, \tfL)$-tangle}} (or in short, a tangle) if:
    \begin{itemize}
    \item either $Z\in \mathcal{M}_{0,2}^*$ and $\ell(Z)\leq \inj$;
    \item {or $Z\in \mathcal{M}_{g,n}^*$ for $2g-2+n=1$ and $\ell^{\mathrm{max}}(\partial Z) \leq \tfL.$}
    \end{itemize}
  \item {We call \emph{tangle in $Z$} any geodesic sub-c-surface of $Z$ which is a tangle.}
  \end{enumerate} 
  \end{defa}
  
  \begin{nota}
    We denote as $\tf[g,n]$ the set of hyperbolic surfaces of signature $(g, n)$ that do not contain
    $(\inj, \tfL)$-tangles, and call them \emph{tangle-free} surfaces.
  \end{nota}

  \subsection{Derived tangles}
  
 The following definition is directly inspired by the work of Friedman~\cite{friedman2003} in the
 context of graphs.
  
  \begin{defa}
  \label{def:dtangle}
  We call $Z\in \mathbf{M}$ a {\em{$(\inj, \tfL)$-derived tangle}} (or, in short, a derived tangle) if
  it is weakly filled by a countable family of $(\inj, \tfL)$-tangles.  We denote by
  ${\mathbf{D}}_{\inj, \tfL} \subset \mathbf{M}$ the set of all $(\inj, \tfL)$-derived tangles.
\end{defa}

In other words, a derived tangle is a c-surface $Z \in \mathbf{M}$ with a certain number of
connected components (some of them circles and some of them hyperbolic manifolds with geodesic
boundary), and such that there is a (possibly finite) sequence of submanifolds $(Z_t)_{t\in I}$
satisfying the following:
\begin{itemize}
\item for all $t \in I$, $Z_t$ is either a simple closed geodesic of length $\leq \inj$ or an
  embedded pair of pants or once-holed torus with all boundary components $\leq \tfL$;
\item $\bigcup_{t \in I} Z_t$ weakly fills $Z$.
\end{itemize}

{\begin{exa}
  Here are a few basic examples of derived tangles:
  \begin{enumerate}
  \item a family of curves $(c_1, \ldots, c_{q})$ of lengths $\leq \inj$;
  \item a family of pairs of pants $(P_1, \ldots, P_{q})$ with all boundary components $\leq \tfL$;
  \item a hyperbolic surface of signature $(0,4)$ with all boundary components of length $\leq \tfL$,
    and containing a simple closed geodesic of length $\leq \tfL$;
  \item a hyperbolic surface of signature $(0,4)$ containing two simple geodesics intersecting exactly twice, one of length $\leq \inj$, the other one separating a pair of pants with all boundary lengths $\leq \tfL$.
  \end{enumerate}
  
  Given that the surface filled by two pairs of pants (or more) can be very complex, depending on
  the relative position of the two pairs of pants, it is very hard to describe all derived tangles.
\end{exa}}

   The following proposition is straightforward:
   \begin{prp} If $X$ is a compact oriented hyperbolic surface, then $X$ contains a maximal derived
     tangle, defined as the geodesic sub-c-surface weakly filled by all the tangles contained
     in~$X$.  We denote it by $\tau_{\kappa, \tfL}(X)$ or just $\tau(X)$.
 \end{prp}

 {
 \begin{exa}
   We have $\tau(X) = (c_1, \ldots, c_{q})$ if and only if $X$ contains exactly $q$ primitive closed
   geodesics of length $\leq \inj$, all of which are simple and disjoint. Similarly,
   $\tau(X) = (P_1, \ldots, P_{q})$ as in the example above if and only if $X$ contains exactly $q$
   hyperbolic pairs of pants with all boundary components of length $\leq \tfL$, and the closures of
   these pairs of pants are disjoint (notably they do not share any boundary components). Note that
   in the latter case, $X$ might actually contain additional tangles: some of the boundary
   components of the pairs of pants could be closed geodesics of length $\leq \inj$ and hence
   tangles themselves.
 \end{exa}}
   
 %  \begin{rem}
   %We caution that $\tau(X)$ may not be a \emph{sub-c-surface} of $X$ in the sense of our forthcoming Definition \ref{d:subsurface}, because several of the boundary components of $\tau(X)$ may be freely homotopic in $X$ (and hence they are equal, since we take the geodesic representative).
   %This is a harmless remark, linked to the fact that in our Definition \ref{d:fill} of \emph{filled surface} we do not allow essential cylinders in $S\setminus E$.
   %\end{rem}
   
  \begin{prp}\label{p:desc_tangles}
    If $Z \in \mathbf{D}_{\inj,\tfL}$ is a $(\inj, \tfL)$-derived tangle, then for any 1d-component
    $c_i$ of $Z$,
    \begin{equation}
      \label{eq:upc}
      \ell(c_i)\leq \inj
    \end{equation}
    and for any 2d-component $\sigma_i$ of $Z$, 
 \begin{align} 
   \label{e:upsigma}
   \ell(\partial \sigma_i) \leq 9 \,\tfL \,\chi(\sigma_i).
 \end{align}
 
% Moreover $Z$ is a disjoint union of tangles and of connected components of higher Euler characteristics, weakly filled by curves of lengths $\leq \tfL$. 
 \end{prp}

{
  \begin{proof}   
    The statement for each 1d-component $c_i$ is obvious: the components $c_i$ are $1$-dimensional
    manifolds and thus cannot contain a surface, so they must be $1$-dimensional tangles, that is to
    say, curves of length $\leq \inj$.
  
    To prove the rest of the proposition, it suffices to treat the case $\mathfrak{c}(Z)=0$ and
    $\mathfrak{q}(Z)=1$, that is to say, the case where $Z$ is a connected hyperbolic surface.  The
    surface $Z$ is weakly filled by an union of tangles $(Z_t)_{t \in I}$, for a set $I = \{1, \ldots, N\}$
    or $I=\N$. We can assume without loss of generality that $I=\N$ by letting $Z_t:=Z_N$ for $t>N$.
   % Let $\tilde Z$ be the surface (strongly) filled by $(Z_t)_{t \in I}$. Then $Z$ is obtained from $\tilde Z$ by attaching cylinders, and
    %$\ell(\partial Z)\leq \ell(\partial \tilde Z)$, so it suffices to prove the result for $\tilde Z$.

    For $t \in \N$, let $S_t\subset  Z$ be the c-surface with geodesic boundary weakly filled by
    $\bigcup_{t' \leq t} Z_{t'}$, and $\chi_t$ be its absolute Euler characteristic. The sequence
    $(S_t)_{t \in \N}$ is a non-decreasing sequence for inclusion.  We decompose the set of times
    $t \in \N$ for which $S_t \neq S_{t+1}$ as two disjoint sets $A, B \subset \N$, where
    $\chi_t \neq \chi_{t+1}$ for $t \in A$ and $\chi_t=\chi_{t+1}$ for $t \in B$.  We bound the
    cardinalities of $A$ and $B$ with two different arguments.
    \begin{itemize}
    \item By additivity of the Euler characteristic and by definition of $A$, the sequence
      $(\chi_t)_{t \in A}$ is an increasing sequence of integers, bounded by $\chi(Z)$: thus,
      $\# A\leq \chi(Z)$.
  \item If $S_t \subsetneq S_{t+1}$ and $\chi_t=\chi_{t+1}$, then $Z_{t+1}$ is a simple closed
    geodesic disjoint from $S_t$. As a consequence, $(Z_{t+1})_{t \in B}$ is a family of disjoint simple
    closed geodesics on $Z$. The cardinality of such a family always satisfies $\# B \leq 3
    \chi(Z)$. 
    \end{itemize}

    Since $\ell(\partial Z_t)\leq 3\tfL$ if $Z_t$ is 2-dimensional, and $\ell(Z_t)\leq \inj$ if $Z_t$
    is 1-dimensional, we have
    $\ell(\partial S_{t+1})\leq \ell(\partial S_{t})+ \max(2\inj, 3\tfL)=\ell(\partial S_{t})+3\tfL$
    for any $t \in A \cup B$. If $t \in B$, we even have
    $\ell(\partial S_{t+1})\leq \ell(\partial S_{t})+ 2\inj$ since $Z_{t+1}$ is 1-dimensional. As
    a consequence, for all~$t$,
    $$\ell(\partial S_{t})\leq 3 \tfL \#A + 2 \inj \# B \leq 9 \tfL \chi(Z)$$ which proves
   that  $\ell(\partial  Z) \leq 9 \,\tfL \,\chi(Z)$ taking $t = \max(A \cup B)+1$.

  \end{proof}
}

 \section{Moebius inversion formula}
 
 \subsection{The ``Moebius inversion formula'' and properties of the Moebius function}
 Let $\inj, \tfL$ be two real numbers with $\inj < \tfL$. Let us additionally assume (mostly for
 simplicity) that $\inj <2\argsh(1)$, so that all closed geodesics of length $\leq \inj$ in a
 hyperbolic manifold are simple and pairwise disjoint \cite[Theorem 4.1.6]{buser1992}. We prove the
 following. 
 
 \begin{thm} \label{t:moebius}There exists a unique function
   $\mu = \mu_{\inj, \tfL} : {\mathbf{M}} \To \R$, such that:
 \begin{itemize}
 \item the restriction of $\mu$ to each $\cM^*_{\vg, \vn}$ is invariant under the action of the full
   diffeomorphism group ({\emph{i.e.}} possibly permuting boundary components and connected
   components);
 \item for every $Z\in {\mathbf{M}} \setminus \mathbf{D}_{\inj,\tfL}$, $\mu(Z)=0$;
\item for every $Z\in {\mathbf{D}}_{\inj, \tfL} $,
\begin{align}\label{e:fundidentity}
1= \sum_{\tau \in \cS(Z)} \mu(\tau).
\end{align}
%(\underbrace{(0,2), \ldots, (0,2)}_{j}, (g_1, n_1), \ldots, (g_m, n_m))\in {\mathfrak{S}}^{(j+m)}$, with $\chi_i=2g_i-2+n_i>0$ for $i=1, \ldots, m$
\end{itemize}
\end{thm}
In the last line, we map $\tau$ to an element of $\mathbf{M}$ to define $\mu(\tau)=\mu(\bar\tau)$, as explained in the Notation \ref{n:sZ}.
 
In addition, the function $\mu$ we construct satisfies the following properties.
First, we can explicit its values in the special case of purely 1d c-surfaces.

\begin{prp}
  \label{prp:circles}
  For $q \geq 1$ and $(\vg, \vn)=((0,2), \ldots, (0,2)) \in (\N_0^2)^q$, if
  $Z=(c_1, \ldots, c_q)\in \cM^*_{\vg, \vn}$, then
  \begin{align}\label{e:circles}
    \mu(Z) =
    \frac{(-1)^{q+1}}{2^q q!} \prod_{i=1}^q \bbbone_{[0,\inj]}(\ell(c_i)).
  \end{align}
\end{prp}
We hence recover the formula \eqref{e:usualIE} obtained in \cite{mirzakhani2013,lipnowski2021}.
We also prove that the Moebius function satisfies the following multiplicativity property.

\begin{prp}
  \label{prp:multi}
  If $Z=(c, \sigma)$ is the decomposition of $Z$ into 1d and 2d parts, then
  \begin{align}\label{e:multi}
    \mu (Z)= -\mu (c) \mu(\sigma).
    \end{align}
\end{prp}

\begin{rem} \label{r:empty}Denoting $\emptyset$ the empty subsurface and using the convention
  $\mu(\emptyset)=-1$, one can conveniently rewrite \eqref{e:fundidentity} as
 \begin{align}\label{e:fundidentity0}
0= \sum_{\tau \in \cS_0(Z)} \mu(\tau)
\end{align}
where $\cS_0(Z)$ is the set of all geodesic subsurfaces of
$Z$, including the empty one. This convention is also compatible with the multiplicative property
\eqref{e:multi}.
\end{rem}

\begin{rem}
  
  In Friedman's work on random graphs, the construction of the Moebius function is very short, and
  there is absolutely no need of any explicit description of $\mu$, whereas for random surfaces \cite{Expo, Ours2}, we
  needed the explicit formula \eqref{e:circles} in the 1d case (or at least the strong decay with
  respect to ${q =} \mathfrak{c}(Z)$), and a rough upper bound in the 2d case (see Theorem
  \ref{thm:boundmu}). There are several reasons for these complications.

  For surfaces, we have to deal with the possibility of closed geodesics of length $\leq \inj$, for
  $\inj$ arbitrarily small, as one of the reasons for proliferation of long closed geodesics.  As a
  consequence, some of the tangles to be removed are circles. Our derived tangles may thus have an
  arbitrary number of connected components, while keeping zero Euler characteristic. On the
  opposite, on a graph, a closed geodesic path always has length $\geq 2$, so there is no need to
  remove short closed geodesics. Tangles in Friedman's work have non-zero Euler characteristic:
  bounding the Euler characteristic of a derived tangle by $M$ automatically implies bounding its
  number of connected components (interestingly, this remark was made by J. Friedman just after
  Lemma 9.6 in \cite{friedman2003}).

  In addition, in \cite[\S 9]{friedman2003}, J. Friedman uses in a crucial way an order relation on
  graphs, such that if two graphs $G, G'$ satisfy $G\leq G'$ then the counting function of geodesics
  on $G'$ is smaller than the counting function of geodesics on $G$. He can then reduce the
  discussion to the set of {\em{minimal tangles}} for this order relation, which is finite for a
  given Euler characteristic. See also \cite{FriedmanKohlerII}. We do not know if an analogue of this
  order relation can be implemented for surfaces.  In the graph case, the upshot is that it is
  obvious that the Moebius function is bounded on the (finite) set of minimal derived tangles of a
  given Euler characteristic; in the surface case, some work is needed to obtain the bound
  presented in the forthcoming Theorem \ref{thm:boundmu}.
\end{rem}

\subsection{Proof of the Moebius inversion formula}

We are now ready to prove the Moebius inversion formula and the properties satisfied by the Moebius
function.  

In the proof, it will be useful to introduce the counting function $N(Z)$ for $Z \in \mathbf{M}$,
\begin{align}
  \label{e:N}
  N(Z):= \#\{\tau\in \cS(Z), \tau \mbox{ diffeomorphic to } Z\}.
\end{align}
This is a topological quantity, independent on the metric: recall that an element of $\cS(Z)$ may be seen as a subset of $Z$ together with additional structure (numbering of connected components, of boundary components, orientation of 1d-components).
Since in \eqref{e:N} $\tau$ coincides with $Z$ as a subset, $ N(Z)$ only counts the possible numberings.
This quantity may be explicitly described in terms of permutations of components having the same
topology, and re-labelling of boundary components, but we will not need a precise formula. Note that
if $Z=(c, \sigma)$ where $c$ is the one-dimensional part of $Z$ and $\sigma$ its two-dimensional
part, then $N(Z)= N(c)N(\sigma)$, simply because a $1$-dimensional manifold cannot contain a
$2$-dimensional one.

The existence and uniqueness of $\mu$ is proven by induction on the pair
$( \chi(Z), \mathfrak{c}(Z))$ with respect to the lexicographic order.  We initialize the recursion
by the case $\chi(Z)=0$, i.e. the purely 1d case, and obtain Proposition \ref{prp:circles} as a
side-product of the proof.

\begin{proof}[Proof of Theorem \ref{t:moebius} in the purely 1d case and Proposition \ref{prp:circles}]
  Let $Z=(c_1, \ldots, c_q)$ be a purely 1-dimensional c-surface. We observe that $N(Z)$ only
  depends on the number of components $q$, is equal to $N_q := 2^q q!$, corresponding to
  permutations and orientation changes of the components~$c_i$.  We use this observation to prove
  the property by induction on the integer $q$, the number of components of the purely 1d c-surface.

  If $Z=(c_1)$ is a single circle of length $> \inj$, then $c_1$ is not a derived tangle and
  $\mu(c_1)=0$. Otherwise, $c_1$ is a derived tangle, in which case \eqref{e:fundidentity} reads
  $1=N_1\mu(c_1)= 2 \mu(c_1)$, so that the formula holds if and only if we define $\mu(c_1) := 1/2$.

  Now suppose that the Moebius formula, as well as the expression \eqref{e:circles}, are known for
  all values $q'<q$, and denote $\mu_{q'}:=(-1)^{q'+1}/(2^{q'} q'!)$. Let $Z=(c_1, \ldots, c_q) $ be
  a purely 1d c-surface with $q$ components. If $Z$ is a not derived tangle (i.e. if one of the
  circles satisfy $\ell(c_i)>\inj$), then we have to have $\mu(Z)=0$. Otherwise, \eqref{e:fundidentity}
  is true if and only if
  \begin{align}\label{e:rankn}
    1 = N_q\mu(Z) + \sum_{1\leq q' <q} \frac{q!}{(q-q')!} 2^{q'} \mu_{q'}.
  \end{align}
  Indeed, the term $q! / (q-q')!$ is the number of ordered subsets of cardinality $q'$ in a set of
  $q$ elements (here the components of $Z$), and the factor $2^{q'}$ corresponds to the number of
  possible orientations of a multicurve with $q'$ components.  The identity \eqref{e:rankn} is true
  if and only if \eqref{e:circles} holds at rank $q$, thanks to the standard binomial identity
  $\sum_{q'=0}^q \frac{q!}{q'!(q-q')!} (-1)^{q'} = 0$.
\end{proof}

We are now ready to prove the result.

 \begin{proof}[Proof of Theorem \ref{t:moebius}]
   We prove Theorem \ref{t:moebius} by induction on the pair of integers
   $( \chi(Z), \mathfrak{c}(Z))$, ordered by lexicographic order. We have already proven the
   statement in the purely 1d case, i.e. when $\chi(Z)=0$.
   Let $Z \in \mathbf{M}$; we assume that we have uniquely defined the Moebius function $\mu$ for
   every $Z' \in \mathbf{M}$ with $\chi(Z') < \chi(Z)$ or $\chi(Z')=\chi(Z)$ and
   $\mathfrak{c}(Z') < \mathfrak{c}(Z)$.

   If $Z \notin \mathbf{D}_{\inj,\tfL}$, i.e. if $Z$ is not a derived tangle, then the Moebius
   function is uniquely defined by letting~$\mu(Z) := 0$, and there is nothing to prove.

   Otherwise, $\mu(Z)$ is uniquely defined by \eqref{e:fundidentity}. Indeed, since $N(Z)>0$, the
   Moebius inversion formula \eqref{e:fundidentity} is equivalent to
   \begin{equation}
     \label{e:mu_def}
     \mu(Z) = \frac{1}{N(Z)} \Bigg(1-\sum_{\substack{Z' \in \cS(Z) \\ Z' \text{ not diffeo to }
         Z}}
     \mu(Z') \Bigg).
   \end{equation}
   Note that any sub-c-surface $Z'$ of $Z$ satisfies $\chi(Z') \leq \chi(Z)$. In case of equality,
   if we write $Z=(c,\sigma)$ and $Z'=(c',\sigma')$, then we must have $\sigma = \sigma'$ as
   manifolds, i.e. $\sigma$ and $\sigma'$ just differ by numbering of the boundary components. This
   forces that $c'$ is a sub-c-surface of $c$ (which is not always the case, as some components of
   $c'$ could a priori have been in $\sigma$). As a consequence, if $Z'$ is not diffeomorphic to $Z$
   and $\chi(Z)=\chi(Z')$, then we must have $\mathfrak{c}(Z')<\mathfrak{c}(Z)$. Hence,
   \eqref{e:mu_def} does uniquely determine $\mu(Z)$ by induction, which is what we needed to
   conclude.
 \end{proof}

 Let us now prove the multiplicative property of the Moebius function $\mu$ that we have announced.

 \begin{proof}[Proof of Proposition \ref{prp:multi}]
   We once again prove this result by strong induction on the pair of integers
   $(\chi(Z),\mathfrak{c}(Z))$ for the lexicographic order. The result is obvious as soon as
   $\chi(Z)=0$ or $\mathfrak{c}(Z)=0$, with the convention $\mu(\emptyset)=-1$.
   
   Let $Z=(c, \sigma)$ be the decomposition of $Z$ into 1d and 2d part, with
   $\mathfrak{c}(Z)\geq 1$. Let us write the identity \eqref{e:fundidentity} in more detail,
   enumerating all elements of $\mathcal{S}(Z)$. We separate them in the following groups.
   \begin{enumerate}
   \item[(a)] Sub-c-surfaces of $Z$ diffeomorphic to $Z$, which contribute a term $\mu(Z) N(Z) $.
   \item[(b)] Sub-c-surfaces $Z'$ of $Z$ entirely contained in the 2-dimensional part $\sigma$ of $Z$
     (which are, in particular, strict sub-c-surfaces, since $\mathfrak{c}(Z) \geq 1$), contributing
     a term $\sum_{Z'\in \cS(\sigma)} \mu(Z')=1$ by the Moebius formula applied to $\sigma$.
   \item[(c)] Strict sub-c-surfaces $Z'$ of $Z$ sharing at least one 1-dimensional component with
     $Z$. Such a sub-c-surface can be decomposed into two groups of components we shall denote as $s$
     and~$\tau$, where $s \in \mathcal{S}(c)$ and $\tau \in \mathcal{S}_0(\sigma)$ (note that this
     is not the decomposition of $Z'$ into 1d and 2d components as $\tau$ itself may have 1d
     components).
   \end{enumerate}

   We notice that the two terms equal to one in \eqref{e:fundidentity}, i.e. its left-hand-side and
   the contribution of case (b), simplify. As a consequence, we obtain that the term (a) and (c) are
   opposite to each other. In detail, this gives:
   \begin{align*}
   1 = \,\, &N(Z)\mu(Z) +\sum_{Z'\in \cS(\sigma)} \mu(Z')\\
   +&\sum_{s\in \cS(c) }
     \sum_{\substack{\tau\in \cS_0(\sigma) \\ \chi(\tau)<\chi(\sigma)}}
     \frac{(\mathfrak{c}(s)+\mathfrak{c}(\tau))!}{\mathfrak{c}(s)! \mathfrak{c}(\tau)!} \mu(s, \tau)
     + \sum_{\substack{s\in \cS(c) \\  \mathfrak{c}(s) < \mathfrak{c}(Z)}}
     \sum_{\substack{\tau\in \cS_0(\sigma)\\ \chi(\tau)=\chi(\sigma)}}  \mu(s, \tau),
     \end{align*}
     so with the announced simplification,
   \begin{equation}
     \label{e:begin}
     N(Z)\mu(Z) = - 
     \sum_{s\in \cS(c) }
     \sum_{\substack{\tau\in \cS_0(\sigma) \\ \chi(\tau)<\chi(\sigma)}}
     \frac{(\mathfrak{c}(s)+\mathfrak{c}(\tau))!}{\mathfrak{c}(s)! \mathfrak{c}(\tau)!} \mu(s, \tau)
     - \sum_{\substack{s\in \cS(c) \\  \mathfrak{c}(s) < \mathfrak{c}(Z)}}
     \sum_{\substack{\tau\in \cS_0(\sigma)\\ \chi(\tau)=\chi(\sigma)}}  \mu(s, \tau).
   \end{equation}
   The combinatorial term
   $(\mathfrak{c}(s)+\mathfrak{c}(\tau))!/(\mathfrak{c}(s)! \mathfrak{c}(\tau)!)$ takes into
   account the fact that the $1$-dimensional components of $\tau$ may be intertwined, in their
   numbering, with the components of $s$. We have separated the terms depending on the value of
   $\chi(\tau)$ as follows.
     \begin{itemize}
     \item If $\chi(\tau)<\chi(\sigma)$ then $\mathfrak{c}(s)$ may take any value between $1$ and
       $\mathfrak{c}(Z)$, leading to the first sum.
     \item If $\chi(\tau)=\chi(\sigma)$, then, as discussed in the proof of Theorem \ref{t:moebius},
       $\tau$ and $\sigma$ just differ in the numbering of their boundary components, and
       $\mathfrak{c}(\tau)= 0$. This sum is therefore restricted to
       $\mathfrak{c}(s) < \mathfrak{c}(Z)$ since it includes only sub-c-surfaces non diffeomorphic
       to~$Z$.
     \end{itemize}     

     Let us call $ \sigma'$ is the $2$-dimensional component of $\tau$. Because we either have
     $\chi( \sigma')<\chi(\sigma)$ or $\mathfrak{c}(s,\tau) < \mathfrak{c}(Z)$, we can use the
     induction assumption and write $ \mu(s, \tau)= - \mu_{\mathfrak{c}(s,\tau)} \mu( \sigma')$. We
     now use the observation that $\mathfrak{c}(s,\tau) = \mathfrak{c}(s)+\mathfrak{c}(\tau)$ and 
     $$\forall (j,k) \in \N_0^2, \quad \frac{(j+k)!}{j!k!} \mu_{j+k} = - \mu_{j}\mu_{k}.$$
     Again by the induction assumption, $ \mu_{\mathfrak{c}(\tau)} \mu( \sigma')=-\mu(\tau)$, so that
     \begin{equation}
       \label{eq:1}
      - \frac{(\mathfrak{c}(s)+\mathfrak{c}(\tau))!}{\mathfrak{c}(s)! \mathfrak{c}(\tau)!} \mu(s,
       \tau)
       = -\mu_{\mathfrak{c}(s)} \mu_{\mathfrak{c}(\tau)} \mu(\sigma')
       =\mu_{\mathfrak{c}(s)} \mu(\tau)= \mu(s)\mu(\tau).
     \end{equation}
   Equation \eqref{e:begin} then becomes
 \begin{align*}
   N(Z)\mu(Z) 
  % & =\sum_{\substack{s \in \cS(c) }}
  %  \sum_{\substack{\tau\in \cS_0(\sigma) \\ \chi(\tau)<\chi(\sigma)}}
  %  \mu(s) \mu(\tau)
  %  + \sum_{\substack{s\in \cS(c) \\ \mathfrak{c}(s) < \mathfrak{c}(Z)}} 
  %  \sum_{\substack{\tau\in \cS_0(\sigma)\\ \chi(\tau)=\chi(\sigma)}  }  \mu(s) \mu(\tau)\\
   & = \Bigg( \sum_{\substack{s \in \cS(c) }} \mu(s) \Bigg)
   \Bigg(\sum_{\substack{\tau\in \cS_0(\sigma) \\ \chi(\tau)<\chi(\sigma)}}
    \mu(\tau) \Bigg)
   + \Bigg(\sum_{\substack{s\in \cS(c) \\ \mathfrak{c}(s) < \mathfrak{c}(Z)}}  \mu(s) \Bigg)
   \Bigg(\sum_{\substack{\tau\in \cS_0(\sigma)\\ \chi(\tau)=\chi(\sigma)}  }  \mu(\tau)\Bigg)
 \end{align*}
 which yields after application of the Moebius inversion formula to each term
 \begin{align*}
   N(Z) \mu(Z)  = 1 \, (0 -N(\sigma)\mu(\sigma))
   + (1-N_{\mathfrak{c}(Z)}\mu_{\mathfrak{c}(Z)} ) \, N(\sigma)\mu(\sigma)
   = - N_{\mathfrak{c}(Z)} N(\sigma) \mu_{\mathfrak{c}(Z)}\mu(\sigma)
 \end{align*}
 which allows us to conclude that $\mu(Z)=-\mu_{\mathfrak{c}(Z)}\mu(\sigma)$ because
 $N(Z)=N_{\mathfrak{c}(Z)}N(\sigma).$
\end{proof}

  \begin{rem} \label{r:multi}Let $(\vg, \vn)\in {\mathfrak{S}}^{q}$. There exists a uniquely defined
    partition $(\Lambda_j)_{1 \leq j \leq  M}$ of $\{1,\ldots, q\}$ into sub-intervals
    such that, for all $j, j' \in \{1, \ldots, M\}$, 
    \begin{align*}
       \forall (i,i') \in \Lambda_j \times \Lambda_{j'},
      \begin{cases}
        (g_i,n_i) = (g_{i'},n_{i'}) & \text{if } j=j' \\
        (g_i,n_i) < (g_{i'},n_{i'}) & \text{if } j < j'.
      \end{cases}
    \end{align*}
    or, in other words, the sequence of signatures is constant on each element $\Lambda_j$ of the
    partition, and (strictly) increasing as we go through the indices $j \in \{1, \ldots, M\}$.

    If $Z= (\sigma_1, \ldots, \sigma_q)\in \cM^*_{\vg, \vn}$, for $j \in \{1, \ldots, M\}$, let
    $Z_j=(\sigma_i)_{i\in \Lambda_j}$ be the collection of components of index $i \in
    \Lambda_j$. This separates the c-surface $Z$ in groups of shared signature. The argument
    yielding \eqref{e:multi} also shows a more general multiplicative property:
 $$\mu(Z)= (-1)^{M+1}\prod_{j=1}^M \mu (Z_j).$$
 \end{rem}

\subsection{Bounds on the values of the Moebius function}

Let us now provide bounds on the values of the Moebius function for purely 2d c-surfaces.

\begin{thm}
  \label{thm:boundmu}
  We can construct explicit increasing sequences $U_1,U_2 : \N \rightarrow \R$ such that, for any
  c-surface $Z \in \mathbf{M}$,
  \begin{align}\label{e:uppermu} |\mu(Z)|\leq \frac{U_1(\chi(Z))}{2^{\mathfrak{c}(Z)}\mathfrak{c}(Z)!}
    \, e^{\tfL U_2(\chi(Z))}.
 \end{align} 
\end{thm}

The definitions of the sequences $U_1$ and $U_2$ will be given in \eqref{e:Fn} and \eqref{e:Gndef}.  We
shall use the following classical estimate, adapted from \cite[Theorem 4.1.6 and Lemma 6.6.4]{buser1992}.
\begin{lem}[{\cite{Ours1}}]
  \label{lem:bound_number_closed_geod}
  Let $Z \in \mathcal{M}_{g,n}^*$, where $2g-2+n>0$. For any $L>0$,
  \begin{equation*}
    \# \{ \text{primitive {non-oriented} closed geodesics of length } \leq L \text{ on } Z\}
    \leq 205 \, {\chi(Z)} \, e^{L}.
  \end{equation*}
\end{lem}
%\laur{j'ai vérifié c'est non-orienté dans la réf et tu l'appliques comme tel après}

\begin{proof}[Proof of Theorem \ref{thm:boundmu}]
  First, we note that, by the multiplicativity property of $\mu$ and the explicit expression in the
  1d-case, we only have to prove \eqref{e:uppermu} in the purely 2d-case, that is when $\mathfrak{c}(Z)=0$. We can therefore procede
  by induction on the absolute Euler characteristic $\chi(Z)$.  First note that if a purely 2d
  c-surface $Z$ is not a derived tangle, then $\mu(Z)=0$ clearly satisfies the desired bound, and
  there is nothing to prove. We therefore restrict the discussion to derived tangles, for which, by
  the Moebius inversion formula, we have:
  \begin{align}\label{e:step} |\mu(Z)|
    &\leq 1 +  \sum_{\substack{\tau \in \cS(Z) \\ \chi(\tau) < \chi(Z) }}|\mu(\tau)|.
   \end{align}
   Further note that the only non-trivial contributions to the right-hand-side of \eqref{e:step}
   come from derived tangles $\tau$ included in $Z$, as $\mu$ vanishes for the other
   $\tau \in \cS(Z)$.
   
   Let us first assume that $Z$ is a purely 2d c-surface of absolute Euler characteristic $1$. In other words,
   $Z = (\sigma_1)$ where $\sigma_1$ is either a once-holed torus or a pair of pants.  The only
   possible derived tangles in $\sigma_1$ that are of Euler characteristic $0$ are the unions of
   simple oriented geodesics on $\sigma_1$ of lengths $\leq \inj$. There are only $3$ non-oriented
   simple geodesics in a pair of pants; if $\sigma_1$ is a once-punctured torus, there can be at most two
   non-oriented simple geodesics of length $\leq\inj$ because we chose $\inj< 2\argsh(1)$. In any
   case, \eqref{e:step} implies that
   \begin{align}\label{e:init}
     |\mu(Z)|\leq 1 +  3\times 2 |\mu_1| +6\times 2^2 |\mu_2| + 6\times 2^3| \mu_3| =8.
    \end{align}
    As a conclusion, for any $Z$ of absolute Euler characteristic $1$, we have $|\mu(Z)| \leq 8$, and the
    claim is true taking $U_1(1)=8$ and $U_2(1)=0$.
   
    Now assume that $Z$ is a derived tangle of absolute Euler characteristic $ \chi(Z) =k$ with
    $k> 1$, and that \eqref{e:uppermu} holds for c-surfaces of absolute Euler characteristic
    $< k$.
   Equation \eqref{e:step} then implies 
    \begin{align}
      |\mu(Z)|\nonumber
      & \leq 1 +
        \sum_{\substack{\tau \in \cS(Z) \\ \chi(\tau) < k }}
      \frac{U_1(\chi(\tau))\, e^{\tfL U_2(\chi(\tau))}}
      {2^{\mathfrak{c}(\tau)}\mathfrak{c}(\tau)!}
      \bbbone_{\tau \in {\mathbf{D}}_{\inj, \tfL}}.
    \end{align}
    By \eqref{e:upsigma} and by monotonicity of the bound on $\mu$, we obtain the rough upper bound
    \begin{equation}
      \label{e:first}
      |\mu(Z)|
      \leq 1 +  U_1(k-1) \,e^{\tfL U_2(k-1)}
      \sum_{j=1}^{k-1}\#\{\tau \in \cS(Z),  \chi(\tau) =j, \ell(\partial \tau)\leq 9\tfL j \}.
    \end{equation}
    We bound the cardinal above using a rough estimate on
    $$\#\{\tau \in \cS(Z), \chi(\tau) =j, \ell(\partial \tau)\leq 9\tfL (k-1) \}.$$ A sub-c-surface
    $\tau \in \cS(Z)$ is fully characterized by the data of the ordered, oriented multicurve
    $\partial \tau= (\gamma_1, \ldots, \gamma_N)$ in $Z$, and of a partition of
    $\{1, \ldots, N\}=\bigsqcup_l \Lambda_l$ into non-empty intervals, such that $m, n$ belong to
    the same $\Lambda_l$ if and only if $\gamma_m$ and $\gamma_n$ border the same connected
    component of $\tau$ on their left.
 %Conversely, given a multicurve $(\gamma_1, \ldots, \gamma_N)$ in $Z$ and a partition of $\{1, \ldots, N\}=\sqcup \Lambda_l$ into non-empty intervals, there exists at most one associated sub-c-surface $\tau \in \cS(Z)$.
    Remark that $N\leq 3k$, because a family of disjoint curves on the surface $Z$ contains at most
    $3 \chi(Z)=3k$ elements.

    Lemma \ref{lem:bound_number_closed_geod} implies that the number of {non-oriented}
    multicurves with $N\leq 3k$ components of length less than $9\tfL (k-1)$ is less that
    $3k(205 k e^{9\tfL (k-1)})^{3k}$. Choices of orientations give an extra factor $2^{3k}$.  This
    gives the very rough upper bound
\begin{align} \label{e:count_sub}
  |\mu(Z)|
  \leq 1 + U_1(k-1)e^{\tfL U_2(k-1)} 3k^2 T(k)(2\times 205 k e^{9\tfL (k-1)})^{3k} 
\end{align}
where $T(k)$ is an upper bound of the number of partitions of a set of cardinality $\leq 3k$, which
leads to our claim 
 \begin{align*}
 |\mu(Z)|& \leq  U_1(k)e^{\tfL U_2(k)}
 \end{align*}
  if we define $U_1, U_2$ by the recursion relations
  \begin{align}
    \label{e:Fn}& U_1(k)=  U_1(k-1) \times 3 k^2 T(k)(410 k)^{3k}+1\\
    \label{e:Gndef}& U_2(k)= U_2(k-1) + 27k (k-1)
  \end{align} 
with the initial values $U_1(1)= 8$ and $U_2(1)=0$.
   \end{proof}
  
 % \begin{rem} The method leading to \eqref{e:count_sub} also shows that, for any $M \geq 1$,
  %\begin{align} \label{e:count_sub2}
    %\sum_{j=1}^{M-1}\#\{\tau \in \cS(Z),  \chi(\tau) =j, \mathfrak{c}(\tau)=0,
    %\ell(\partial \tau)\leq 9\tfL j \}  \leq 3M^2 T(M)(205 \chi(Z) e^{9\tfL (M-1)})^{3M} .
%\end{align}
  %\end{rem}

  \section{Counting topological types of geodesics in tangle-free surfaces}

Let us now prove that the tangle-free hypothesis does prevent exponential proliferation of closed
geodesics, at scales which are constant multiples of $\tfL$.
  
\subsection{Local topological types} \label{s:ltt} Recall from \cite{Ours1} the notion of
\emph{local topological type} of a loop (an oriented closed path). We omit some of the notations
from \cite{Ours1} which are not used in the present paper; the reader can consult the full article
for additional details.

\begin{defa}
  A \emph{filling type} is a pair $(g_{\ft},n_{\ft}) \in \mathfrak{S}^1$. For each
  filling type $(g_{\mathbf{S}},n_{\mathbf{S}})$, let us \emph{fix} a  smooth oriented surface
  $\mathbf{S}$ of signature $(g_{\ft},n_{\ft})$.
\end{defa}

  \begin{defa}
  \label{def:local_type}
  A \emph{local loop} is a pair $(\ft,\ct)$, where $\ft$ is a filling type and $\ct$ is a primitive
  loop filling $\ft$. Two local loops $(\ft,\ct)$ and $(\ft',\ct')$ are said to be \emph{locally
    equivalent} if $\ft=\ft'$ (i.e.  $g_{\ft} = g_{\ft'}$ and $n_{\ft}=n_{\ft'}$), and there exists
  a positive homeomorphism $\psi : \ft \rightarrow \ft$, possibly permuting the boundary components
  of $\ft$, such that $\psi \circ \ct$ is freely homotopic to $\ct'$. This defines an equivalence
  relation $\eq$ on local loops. Equivalence classes for this relation are denoted as
  $\mathbf{T} = \eqc{\ft, \ct}$ and called \emph{local (topological) types} of loops.
\end{defa}

Let us now define a notion of local topological type for loops on a hyperbolic surface
$Z \in \cM_{g,n}^*$, where $2g-2+n>0$.

\begin{nota}
  Let $\gamma$ be a loop on a hyperbolic surface $Z \in \cM_{g,n}^*$. We denote by
  $\ft(\gamma)\subset Z$ the \emph{surface filled by} $\gamma$ in $Z$, i.e. the minimal embedded
  surface in $Z$ containing $\gamma$ (up to isotopy).
\end{nota}

The surface $\ft(\gamma)$ is constructed in \cite[\S 4]{Ours1} by taking a regular
neighbourhood~$\cN$ of $\gamma$, to which we add the contractible components of $Z\setminus \cN$ if
there are any. Note that, because we do not add cylinders in the complement of $\cN$, the filled
surface is not necessarily a sub-c-surface of $Z$ according to our conventions (we would need to
take the surface \emph{weakly filled} by $\gamma$ for this property to hold).
%\laur{changé $\cR$ en $\cN$ comme $\cR$ est pris après mais $\cN$ libre}

\begin{defa}
  Let $\eqc{\ft,\ct}$ be a local topological type.  A loop $\gamma$ in a hyperbolic surface~$Z$ is
  said to \emph{belong to the local topological type}~$\eqc{\ft,\ct}$ if there exists a positive
  homeomorphism $\phi : \ft(\gamma) \rightarrow \ft$ such that the loops $\phi \circ \gamma$
  and $\ct$ are freely homotopic in~$\ft$.

  Two loops $\gamma_1$, $\gamma_2$ in two hyperbolic surfaces $Z_1$, $Z_2$ are said to be
  \emph{locally equivalent} if they belong to the same local topological type.
\end{defa}

\subsection{Statement of the counting result}

We are now ready to state our counting result.
\begin{defa} Take $0<\inj< \tfL$, $L>0$ and $\ft$ a filling type.

  We denote by $\mathrm{Loc}_{\ft}^{\inj,\tfL, L}$ the set of local topological types
  $\eqc{\ft, \ct}$ such that there exists a $(\inj, \tfL)$-tangle-free surface $Z$ and a
  homeomorphism $\phi: \ft\rightarrow Z$ such that $\phi(\ct)$ is freely homotopic to a geodesic of
  length $\leq L$ on~$Z$.

If $\chi$ a positive integer, we define
$$\mathrm{Loc}_{\chi}^{\inj,\tfL, L}
:= \bigsqcup_{\substack{\ft \text{ filling type} \\ \chi(\ft)\leq \chi}}\mathrm{Loc}_{\ft}^{\inj,\tfL, L}$$ 
where the union is over the (finite) set of filling types $\ft$ with $\chi(\ft) = 2g_{\ft}-2+n_{\ft}\leq \chi$.
\end{defa}

Obviously, the set $\mathrm{Loc}_{\chi}^{\inj,\tfL, L}$ is decreasing as a function of $\inj$ and
$\tfL$, and increasing as a function of~$L$.  The rest of this article is devoted to proving that
the cardinality of $\mathrm{Loc}_{g}^{\inj,\tfL, L}$ is at most polynomial in $L$ if $\inj>0$ is
fixed and the ratio $L/\tfL$ is bounded.

 \begin{thm}
  \label{t:TF_curves}
  Let $\ft$ be a filling type. There exists a constant $C_{\ft}>0$ such that, for any
  $0 < \inj < \tfL$ and $L>0$,
    \begin{equation*}
      \# \mathrm{Loc}_{\ft}^{\inj,\tfL, L} \leq C_{\ft} \,
      \frac L{\tfL}
      \left(3\frac{(4L)^2}{\inj^2} +3\right)^{ 6 L(6g_{\ft}-6   + 2n_{\ft}) /\tfL}
      \left(\frac{L}{\inj} + 1\right)^{6 L/\tfL}.
  \end{equation*}
\end{thm}

\begin{cor}
  Let $\chi$ be a positive integer. There exists a constant $C_\chi > 0$ such that, for any
  $0 < \inj < \tfL$ and $L>0$,
    \begin{equation*}
    \# \mathrm{Loc}_{\chi}^{\inj,\tfL, L} \leq C_\chi \,  \frac L{\tfL}
      \left(3\frac{(4L)^2}{\inj^2} +3\right)^{18\chi L /\tfL}
      \left(\frac{L}{\inj}+1\right)^{6 L/\tfL} .
  \end{equation*}
  \end{cor}
  This directly implies Theorem \ref{t:TF_curves-intro}. When applying this result to random
  hyperbolic surfaces, the case $\tfL=\alpha \log g$ and $L=A\log g$ is especially relevant; for this
  choice of parameters, we have the following corollary:
\begin{cor}
  \label{cor:TF_curves}
  Let $\chi\geq 1$, $\inj, \alpha \in (0,1)$ and $A \geq 1$. Then,
    \begin{equation*}
      \# \mathrm{Loc}_{\chi}^{\inj,\alpha \log g, A\log g} = \O[\inj, \alpha, A,\chi]{(\log g)^{\beta_{\inj,\alpha,A, \chi}}}
  \end{equation*}
  for a $\beta_{\inj, \alpha, A, \chi} > 0$ depending only on $\inj$, $\alpha$, $A$ and $\chi$.\end{cor}

\subsection{Counting simple curves in surfaces}
Our counting argument relies mostly on the following estimate, based on Thurston's description of the space of simple curves, as exposed in~\cite{FLPf}. We denote by $i(\gamma_1, \gamma_2)$ the geometric intersection 
of two loops $\gamma_1, \gamma_2$ in minimal position 

In all the forthcoming discussion, $S$ is a fixed compact surface of signature $(g_S, n_S)$.

\begin{nota}
  Let $\lambda=(\lambda_1, \ldots, \lambda_{3g_S-3 +n_S})$ be a multicurve cutting $S$ into surfaces
  of Euler characteristic $-1$. Let $N, I, J$ be three positive integers. Call
  $\cC_\lambda(N, I, J)$ the set of $N$-tuples of simple curves $(\gamma_1, \ldots, \gamma_N)$ in
  minimal position, such that
  \begin{equation*}
    \begin{cases}
      i(\gamma_j, \lambda_k)\leq I
      & \text{for all } j \in \{1, \ldots, N\} \text{ and }  k \in \{1, \ldots, 3g_S-3+n_S\} \\
      i(\gamma_j, \gamma_k)\leq J
      & \text{for all distinct } j, k \in \{1, \ldots, N\}.
    \end{cases}
  \end{equation*}
\end{nota}

{\begin{defa}
  We define $\overline{\cC}_\lambda(N, I, J)$ to be the quotient of $\cC_\lambda(N, I, J)$ by the
  equivalence relation:
\begin{align}\label{e:r}
  (\gamma_1, \ldots, \gamma_N) ~ \cR ~ (\gamma'_1, \ldots, \gamma'_N)
  \quad \Leftrightarrow \quad
  \parbox{8cm}{$\exists \phi : S \rightarrow S$ positive homeomorphism:
  \\ $\phi(\gamma_j)$  homotopic to $\gamma_j'$ for all $j \in \{1, \ldots, N\}$.}
\end{align}
\end{defa}
}

\begin{prp}\label{p:fundamental_count} With the notations above,
\begin{align}\label{e:boundCNIJ}
  \# \overline{\cC}_\lambda(N, I, J) \leq
  2^{ 3g_S-3 +n_S}\Big( 3(I+1)(J+1)\Big)^{N(3g_S-3   +n_S)}.
\end{align}
\end{prp}

{
\begin{rem}
  We shall actually prove the upper bound \eqref{e:boundCNIJ} for equivalence classes w.r.t. a finer
  equivalence relation than $\mathcal{R}$, obtained adding the requirement that $\phi$ is in the
  group generated by Dehn twists around curves in $\lambda$ in \eqref{e:r}. 
\end{rem}}

The proof follows standard arguments borrowed from Thurston's description of the space of simple
curves. We rely on the reference \cite[Expos\'e 4]{FLPf}.
 
\begin{proof}
%We are not trying to be as optimal as in \cite{FLPf}, in particular, our bound could probably be improved to obtain an upper bound in $\cO( I^{3g_S-3 +n_S } (2J+1)^{3g_S-3 +n_S})^N)$ (that is to say, a lower power of $I$).
  To an element $(\gamma_1, \ldots, \gamma_N)\in \cC_\lambda(N, I, J)$ we associate indices 
  $$m^j(k):=i(\gamma_j, \lambda_k), \quad \text{where }j \in \{1, \ldots, N\}, k \in \{1, \ldots, 3g_S-3+n_S\}.$$
  If $B$ is a subset of $\{1, \ldots, 3g_S-3 +n_S\}$, call $\cC_{\lambda, B}(N, I, J)$ the set of
  elements of $\cC_\lambda(N, I, J)$ such that $m^j(k)=0$ for all $j \in \{1, \ldots, N\}$ and
  $k\in B$. Let $\overline{\cC}_{\lambda, B}(N, I, J)$ be the set of corresponding equivalence
  classes. We shall show that, for any fixed subset $B$ of $\{1, \ldots, 3g_S-3+n_S\}$,
  $$\# \overline{\cC}_{\lambda, B}(N, I, J)
  \leq \Big( 3(I+1) (J+1) \Big)^{N(3g_S-3 +n_S-\# B)}.$$ The upper bound \eqref{e:boundCNIJ} is then
  obtained by summing over all possible subsets $B$.

  Let us fix a subset $B$ of $\{1, \ldots, 3g_S-3 +n_S\}$ and describe a ``normal form'' for a
  representative of the equivalence class of any
  $(\gamma_1, \ldots, \gamma_N)\in \cC_{\lambda, B}(N, I, J)$.

  Let $Q_1, \ldots, Q_{2g_S-2+n_S}$ be the connected components of $S\setminus \lambda$.  Fix
  $j\in \{1, \ldots, N\}$.  It is classical that the data of the family $(m^j(k))_{k\not\in B}$
  determines exactly one topology for each restriction $\gamma_j\cap Q_m$, where
  $m \in \{1,\ldots, 2g_S-2+n_S\}$.  To exploit this fact, we follow quite closely the discussion of
  \cite[Expos\'e 4, \S III]{FLPf} but introduce some small variants, suited to our counting problem.

  Let $P^2$ be a fixed pair of pants with boundary components numbered $B_1, B_2, B_3$, For each
  triple of non-negative integers $(M_1, M_2, M_3)$ such that $M_1+M_2+M_3$ is even, there exists a
  model of multiple arc in $P^2$ having $M_1, M_2, M_3$ intersections with each of the three
  boundary components.  The list of models is given in \cite[Expos\'e 4, \S III, Figure 3]{FLPf}.
%For any $(M_1, M_2, M_3)$, and for $i=1, 2, 3$, we number from $1$ to $M_i$ the points of the model lying on the boundary component $i$, respecting the orientation of the boundary. 
%In addition, for every pair of boundary components $i, j \in \{1, 2, 3\}$ ($i\not =j$), we fix an arc $\tau_{ij}$ joining the two boundary components $B_i$ and $B_j$, and disjoint from all models of multiple arcs.
%The boundary component $B_i$ is touched by two arcs $\tau_{ij}, tau_{ik}$; we select one of them (for instance $\tau_{ij}$ if $j<k$) and call it the ``pink arc'' arriving at $B_i$.

  We denote by $\lambda_k\times [-1, +1]$ some tubular neighbourhoods of each $\lambda_k$, chosen to
  be pairwise disjoint in $S$. The complement of their union is made of disjoint pair of pants
  $R_1, \ldots, R_{2g_S-2+n_S}$, with $R_m\subset Q_m$. Let us fix once and for all a homeomorphism
  $\phi_m$ between each $R_m$ and the reference pair of pants $P^2$, thus obtaining a collection of
  models of multiple arcs in $R_m$.
%In each $R_k$, we also have a pink arc  arriving at each boundary component.

  The data of the family $(m^j(k))_{k\not\in B}$ determines a model of multiple arc in each $R_m$,
  such that $\gamma_j\cap R_m$ is isotopic to the model.
%Recall that the endpoints of the model on $\partial R_m$ are numbered,
%so on each sides of the cylinder $\lambda_k\times [-1, +1]$ we have two families of points labelled $1, \ldots, m^j(k)$: $x^-_1, \ldots, x^-_{m^j(k)}$ on  $\lambda_k\times \{-1\}$, $x^+_1, \ldots, x^+_{m^j(k)}$ on  $\lambda_k\times \{+1\}$.

  There exists a curve $\gamma'_j$ isotopic to $\gamma_j$, that coincides with the model in each
  $R_m$ (furthermore, by \cite[Expos\'e 4, Lemme 5]{FLPf}, any two curves $\gamma_j'$, $\gamma''_j$
  satisfying this requirement are isotopic in each cylinder $\lambda_k\times [-1, +1]$, for an
  isotopy fixing the boundary $\lambda_k\times \{-1, +1\}$). The ``normal form'' $\gamma'_j$ is
  fully determined by the data of $(m^j(k))_{k\not\in B}$ and of the isotopy class of {the
    restrictions} ${\gamma'_j}\restr _{\lambda_k\times [-1, +1]}$ for all $k\not\in B$.

  Let {$S^j_k, T^j_k$} be two distinct transversals to this cylinder, joining
  $\lambda_k\times \{-1\}$ and $\lambda_k \times \{+1\}$, sharing the same endpoints and
  intersecting only at those endpoints. The isotopy class of
  ${\gamma_j}\restr _{\lambda_k\times [-1, +1]}$ is fully parametrized by {the integers}
  $s^j_k$ and $t^j_k$, defined as the geometric intersection of
  ${\gamma_j}\restr _{\lambda_k\times [-1, +1]}$ respectively with {$S^j_k$ and $T^j_k$}.
  Besides, we have either $m^j(k)=s^j_k+t^j_k$, or $s^j_k=m^j(k)+t^j_k$ or $t^j_k=m^j(k)+s^j_k$, see
  \cite[Expos\'e 4, \S III, Figure 5]{FLPf}.

  The following observation is crucial: applying a Dehn twist to the transversal {$S^j_k$}
  while keeping the parameters $s^j_k$ constant amounts to applying the same Dehn twist on
  $\gamma_j$. Since we are only counting equivalence classes, it is enough to show that there exists
  a choice of transversal {$S^j_k$} for which we can bound all the $s^j_k$ by the integer $J$.

  Given $k\not\in B$, we choose (for instance) the smallest $j_o\in \{1, \ldots, N\}$ such that
  $\gamma_{j_o}$ intersects $\lambda_k$, and we choose {$S^j_k$} to be a branch of
  ${\gamma'_{j_o}}\restr_{\lambda_k\times [-1, +1]}$.  Our assumptions on the intersection numbers
  imply that $0\leq s^j_k\leq J$ and $0\leq m^j(k)\leq I$ for $k\not\in B$. As a consequence, for
  each $1 \leq j \leq N$ and $k \notin B$, there are at most $J+1$ possible values that the integer
  $s_k^j$ can take, and at most $I+1$ values that the integer $m^j(k)$ can take. Taking into account
  the three situations $m^j(k)=s^j_k+t^j_k$ or $s^j_k=m^j(k)+t^j_k$ or $t^j_k=m^j(k)+s^j_k$, we
  obtain the claimed upper bound.

 % \laur{J'ai remplacé les $\rceil$ par des $|$ (ce sont bien des restrictions?). Si c'est autre
  %  chose que des restrictions il faut expliciter la notation je pense. C'est la macro $\restr$.}
\end{proof}

\subsection{Sequence of volutes associated with a non-simple loop}
\label{s:voluteconst}
Let $\gamma$ be a non-simple (oriented) loop in minimal position, contained in a surface $S$. The following construction is inspired by a discussion from Luo--Tan \cite[\S 1.2]{luotan2014}, and relies on the fact that the complement of~$\gamma$ does not contain any bigon.

%\laur{j'ai rajouté des subsubsections}

\subsubsection{Construction of one volute}

Let $x_0$ be a point on $\gamma$, assumed not to be a self-intersection point of $\gamma$, and choose a parametrization $\gamma:\R\To S$ such that $\gamma(0)=x_0$. Consider the path $A^-_t=\gamma([-t, 0])$ for $t>0$. Let $t_->0$ be the smallest positive number such that $A_{t_-}$ is not a simple arc (say, $\gamma(-t_-)=\gamma(t')$ with $t'\in (-t_-, 0]$).
%Either $\gamma(t_+)\in \gamma([-t_+, t_+))$, or $\gamma(-t_+)\in \gamma((-t_+, t_+])$: assume for instance that we are in the first case. Say, $\gamma(t_+)=\gamma(t')$ with $t'\in [-t_+, t_+)$.
Next, let $t_+>0$ be the smallest positive time so that $\gamma(t_+)\in \gamma([-t_-, t_+))$ (say,
$\gamma(t_+)=\gamma(t'')$ with $t''\in [-t_-, t_+)$).

%\laur{J'ai remonté la def.}

\begin{defa}
 We call $G(x_0):=\gamma([-t_-, t_+])$ the \emph{volute} associated with $x_0$.
 \end{defa}

 {The volute $G(x_0)$ is a graph of Euler characteristic $-1$. If we call $\Sigma(x_0)$ a regular
 neighbourhood of the volute, then it is either a three holed sphere or a once-holed torus. } The
graph $G(x_0)$ may be decomposed into two rooted simple oriented loops
$g_1(x_0)= \gamma_{\restr [-t_-, t']}$ {and} $g_2(x_0)= \gamma_{\restr [t'', t_+]}$ (with
respective roots $\gamma(-t_-)$ and $\gamma(t_+)$), and a simple path $b(x_0)= \gamma([t', t''])$
joining the root of $g_1(x_0)$ to the root of $g_2(x_0)$.

\begin{figure}[h!]
  \centering
  \begin{subfigure}[b]{0.2\textwidth}
    \centering
    \includegraphics[scale=0.65]{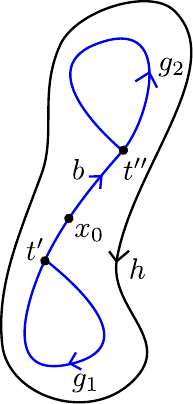}
    \caption{$g_1bg_2^{-1}b^{-1}$.}
    \label{fig:fig_a}
  \end{subfigure}% 
  \begin{subfigure}[b]{0.2\textwidth}
    \centering
    \includegraphics[scale=0.65]{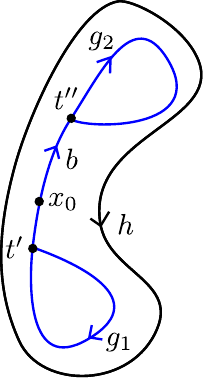}
    \caption{$g_1bg_2b^{-1}$.}
    \label{fig:fig_b}
  \end{subfigure}% 
  \begin{subfigure}[b]{0.25\textwidth}
    \centering
    \includegraphics[scale=0.65]{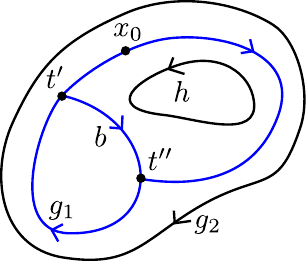}
    \caption{$g_1bg_2^{-1}b^{-1}$.}
    \label{fig:fig_c}
  \end{subfigure}% 
  \begin{subfigure}[b]{0.35\textwidth}
    \centering
    \includegraphics[scale=0.65]{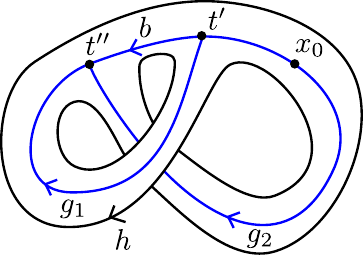}
    \caption{$g_1bg_2^{-1}b^{-1}g_1^{-1}bg_2b^{-1}$.}
    \label{fig:fig_d}
  \end{subfigure}% 
    \caption{Four cases for the volute $G(x_0)$ and the expression for $h$ in each case.}
    \label{f:volutes}
\end{figure}

In Figure \ref{f:volutes}, {we represented all possible volutes, after assuming for symmetry
  reasons} that at time $t'$ the curve $\gamma$ crosses itself from the right. The picture shows the
four situations that arise, depending on whether $t'\leq t''$ or $t''>t'$, and whether at time $t''$
the curve $\gamma$ crosses itself from the right or from the left. In cases (a), (b), (c),
$\Sigma(x_0)$ is a pair of pants, its boundary is made of two disjoint simple curves
$\gamma_1(x_0), \gamma_2(x_0)$ homotopic respectively to $g_1(x_0)$, $g_2(x_0)$, and of a third
simple curve $h(x_0)$, homotopic to
\begin{align}\label{e:h'}\hbis(x_0)=g_1(x_0) \smallbullet b(x_0)\smallbullet g_2(x_0)^{\pm
    1}\smallbullet b(x_0)^{-1},
\end{align}
%\laur{changé $h'$ en $\tilde{h}$, autre possibilité $h_0$? macro $\hbis$}
where $\smallbullet$
means the concatenation of paths. In case (d), $\Sigma(x_0)$ is a once-holed torus: the boundary of
$\Sigma(x_0)$ is made of only one curve $h(x_0)$ homotopic to
$$g_1(x_0) \smallbullet b(x_0)\smallbullet g_2(x_0)^{-1}\smallbullet b(x_0)^{-1} g_1(x_0)^{-1}
\smallbullet b(x_0)\smallbullet g_2(x_0) \smallbullet b(x_0)^{-1}.$$ The surface $\Sigma(x_0)$ is
filled by two simple curves $\gamma_1(x_0), \gamma_2(x_0)$ intersecting once, homotopic respectively
to $g_1(x_0), g_2(x_0)$.  To fix ideas we can decide to put their intersection point at
$\gamma(t')$.

\subsubsection{Sequence of volutes}
 
We now take $\eps>0$ very small, so that $\gamma([t'', t_++\eps|)$ still intersects exactly once, at $\gamma(t_+)$. Take as new origin $x_1:=\gamma(t_++\eps)$ and repeat the same construction with $x_0$ replaced by $x_1$. We obtain a volute $G(x_1)$ formed of $g_1(x_1), g_2(x_1), b(x_1)$, and a surface $\Sigma(x_1)$ of Euler characteristic $-1$ containing three simple curves $(\gamma_1(x_1), \gamma_2(x_1), h(x_1))$.
If $\Sigma(x_1)$ is a pair of pants then $(\gamma_1(x_1), \gamma_2(x_1), h(x_1))$ are disjoint simple curves bordering $\Sigma(x_1)$. If $\Sigma(x_1)$ is a once-holed torus then $h(x_1)$ is homotopic to the boundary, and $\gamma_1(x_1), \gamma_2(x_1)$ are simple curves intersecting once and filling $\Sigma(x_1)$.
Remark that the rooted oriented loops $g_2(x_0)$ and $g_1(x_1)$ coincide, so we can impose $\gamma_2(x_0)=\gamma_1(x_1)$.

Repeating this construction, we obtain a sequence of volutes $(G(x_n))_{{n \geq 0}}$,
{satisfying} $g_2(x_n)=g_1(x_{n+1})$, {together with} a surface $\Sigma(x_n)$ of Euler characteristic
$-1$ containing three simple curves $(\gamma_1(x_n), \gamma_2(x_n), h(x_n))$, with
$\gamma_2(x_n)=\gamma_1(x_{n+1})$.
%The third curve $ h(x_n))$ is homotopic to the boundary, and $\gamma_1(x_n), \gamma_2(x_n) $ are simple curves disjoint from $ h(x_n)$.
Since $\gamma$ is a loop, there is $N$ such that $G(x_N)=G(x_0)$ and we stop the construction. 
%At step $n$ of the construction, we are either in Case A or in Case B. We refer to the first possibility as Case A($n$), and to the second one as Case B($n$).

Keep in mind that $h(x_n)$ is homotopic to $\hbis(x_n)$, formed like in \eqref{e:h'} by a
concatenation of $g_1(x_n)$, $g_2(x_n)$ and copies of $b(x_n)$. As a consequence of the
construction, for any Riemannian metric on $S$, we have the following length inequalities.
%\laur{grouped in a lemma}
\begin{lem}
  For any $n \geq 0$, 
  \begin{align}\label{e:boundGn} \ell(G(x_n)) \geq \max \left( \ell(g_1(x_n)), \ell(g_2(x_n)),
    \frac{
    \ell(\hbis(x_n)) }2 \right).
  \end{align}
  Furthermore,
  \begin{align}\label{e:bound_volutes}\ell(\gamma)\leq \sum_{n=0}^{N-1} \ell(G(x_n))\leq 3
    \ell(\gamma).
  \end{align}
\end{lem}

%\laur{mis dans un environnement nota}
We now introduce more compact notations.

\begin{nota}
  For an index $n \in \Z \diagup N\Z$, we write $g_n=\gamma_2(x_n)=\gamma_1(x_{n+1})$, $h_n= h(x_n)$
  and $\Sigma_n=\Sigma(x_n)$.
\end{nota}

If $\Sigma_n$ form a pair of pants, {then} $(g_{n-1}, g_n, h_n)$ are its boundary
components. If $\Sigma_n$ is a once-holed torus, {then} $h_n$ is its boundary and
$g_{n-1}$, $g_n$ are two simple curves, disjoint from $h_n$, intersecting once and filling $\Sigma_n$.

We take the collection of curves $((g_{n-1}, g_n, h_n))_{n\in \Z \diagup N\Z}$ in minimal position.
Remark that if both $\Sigma_n$ and $\Sigma_{n+1}$ are once-holed tori, there are constraints on the relative position of the curves $g_{n-1}, g_n$ and
$g_{n+1}$ (for instance, $g_{n+1}$ must intersect $g_{n-1}$). However, since we just need to find upper bounds on the number of
topological types, it will not be necessary to determine all conditions fulfilled by the curves
$g_n$.
 %If $\Sigma_n$ is a once-holed torus, let us draw a transversal $V_n$ going from $h_n$ to itself, disjoint from $g_{n-1}$ and intersecting $g_n$ once.

Remark that our construction is purely topological, in the sense that is does not depend on the choice of a metric, but only on the choice
of a representative of the homotopy class of $\gamma$.
 
\subsubsection{The segments $I_n$ and Condition (V)}

{Let us introduce a family of segments $I_n = [O_n,T_n]$ indexed by
  $n \in \Z \diagup N \Z$. For an index $n$, we pick the segment $I_n$ to be:
\begin{itemize}
\item a segment going from $g_{n-1}$ to $ g_{n}$ inside $\Sigma_{n}$ if $\Sigma_n$ is a pair of
  pants;
\item the segment reduced to a point, the intersection point between $g_{n-1}$ and $g_{n}$, if
  $\Sigma_n$ is a once-holed torus.
\end{itemize}
We note that the endpoints $T_{n}$ and $O_{n+1}$ belong in $g_{n}$ for any index $n$. The point
$O_{n+1}$ will serve as base-point of $g_n$.
}

If $\Sigma_n$ and $\Sigma_{n+1}$ are pairs of pants obtained by the construction above, they must satisfy the following condition,
referred below as Condition (V): $I_n$ and $I_{n+1}$ both lie on the same side of $g_n$.
This condition depends on the pairs of pants $\Sigma_n$ and $\Sigma_{n+1}$ and on the choice of labelling of their boundary components, but not on the choice of $I_n$ and $I_{n+1}$, provided $I_n$ goes from 
$g_{n-1}$ to $ g_{n}$ and  $I_{n+1}$ goes from 
$g_{n}$ to $ g_{n+1}$. It
mostly serves to give a meaning to the forthcoming discussion.
% We write $I_{n+1}=[O_{n+1}, T_{n+1}]$ where $T_n, O_{n+1}\in g_n$ for all $n$. 
%, in such a way that $[O_{n+1}, T_{n+1}]$ does not intersect $[O_n, T_n]$ -- except in the {\em{special case}} below;

Note that the prescriptions above only define $I_{n+1}$ is modulo homotopy with gliding endpoints
along $g_{n}, g_{n+1}$, so we can impose further conditions.  
\begin{enumerate}
\item Assume first that both $\Sigma_n$ and $\Sigma_{n+1}$ are pairs of pants.
  \begin{itemize}
  \item If $I_{n+1}$ intersects $h_n$, we can impose that $O_{n+1}\not= T_n$ (by gliding $O_{n+1}$
    along $g_n$) and that $I_n$ and $I_{n+1}$ are in minimal position (i.e. have minimal intersection number in a gliding homotopy class).
  \item If $I_{n+1}$ intersects $g_{n-1}$, we impose that $O_{n+1}=T_n$ and {that the open
      segment} $(O_{n+1}, T_{n+1})$ contains $(T_n, O_n)$. Note that in this case we can have $T_{n+1}=O_n$; this happens if $g_{n+1}=g_{n-1}$ and $h_{n+1}=h_n$ (which implies, in
  particular, that $\Sigma_{n+1}=\Sigma_{n}$).  
  \end{itemize}
  \item If $\Sigma_n$ is a pair of pants and $\Sigma_{n+1}$ a once-holed torus, then we impose that
  $I_n$ is in minimal position with $g_{n+1}$.
\item If $\Sigma_n$ is a once-holed torus and $\Sigma_{n+1}$ a pair of pants, then we impose that
  $g_n$ is in minimal position with $I_{n+1}$.
\end{enumerate}

We let $c_n$ be the smallest subpath of $g_n$, oriented in the direction of $g_n$, and going from $T_n$ to $O_{n+1}$ (in case $T_n\not =O_{n+1}$, this is a strict subsegment of $g_n$; in case $T_n=O_{n+1}$,
it is a trivial subpath). Call 
\begin{align}\label{e:Gn} p_n= g_{n-1}\smallbullet {I_n} \smallbullet c_n  
  ,\end{align} a continuous path going from $O_{n-1}$ to $O_n$. Then {$(p_{n})_{n \in \Z
    \diagup N \Z}$ is compatible in the following sense.} 

\begin{rem} A finite sequence of continuous paths $p_n: [0, t_n]\To S$ (indexed by
  $n \in \Z \diagup N \Z$) is said to be compatible if $p_n(t_n)=p_{n+1}(0)$.  We define the
  homotopy relation between two sequences of $N$ compatible paths $(p_{n})_{n \in \Z \diagup N \Z}$
  and $(q_{n})_{n \in \Z \diagup N \Z}$ as the usual homotopy relation between $p_n$ and $q_n$ for
  each $n \in \Z \diagup N \Z$, with the additional condition that the deformation must preserve the
  compatibility condition.
 \end{rem}

 Coming back to \eqref{e:Gn}, there exists integers $k_n \geq 0$, $n\in \Z \diagup N\Z$, such that
 the loop $\gamma$ is freely homotopic to
\begin{align}\label{e:desc} \smallbullet_{n=0}^{N-1} (p_n \smallbullet g_n^{k_n}).
\end{align}
If both $\Sigma_n$ and $\Sigma_{n+1}$ are once-holed tori then $k_n=0$. If $\Sigma_n$ or
$\Sigma_{n+1}$ is a pair of pants then any loop homotopic to $\gamma$ contains at least $k_n-1$
disjoint copies of a closed curve homotopic to $g_n$.

 \subsection{A surjective map}

 Below we describe a surjective map, from a set $\cP$ onto the set of free homotopy classes of
 non-simple loops in $S$. We later identify a subset $\cP^{\inj,\tfL, L}\subset \cP$ that is mapped
 onto a set of representatives of $ \mathrm{Loc}_{S}^{\inj,\tfL, L}$. Our upper bound on
 $ \#\mathrm{Loc}_{S}^{\inj,\tfL, L}$ will be obtained by evaluating the cardinality of
 $\cP^{\inj,\tfL, L}$, modulo homeomorphisms and homotopy.

 {
 \begin{nota}
   For any integer $N\geq 1$, denote by {$\cP_N$} the set of sequences
   $(P_n)_{n\in \Z/ N\Z}$, where:
   \begin{itemize}
   \item $P_n=(g_n, h_n)$ is a collection of two simple curves in $S$, with $g_n$ oriented;
   \item for every $n \in \Z \diagup N \Z$, we have one of the following:
     \begin{enumerate}
     \item either $(g_{n-1}, g_n, h_n)$ are disjoint and border a pair of pants $\Sigma_n$;
     \item or $g_{n-1}, g_n$ intersect once and fill a once-holed torus $\Sigma_n$ with
       boundary~$h_n$;
     \end{enumerate}
   \item if $\Sigma_n$ and
 $\Sigma_{n+1}$ are pairs of pants, we further impose that they satisfy Condition~(V).   \end{itemize}
 \end{nota}

 We then group together all possible values of $N$ as well as the data of the integers
 $(k_n)_{n \in \Z \diagup N \Z}$ in a big set $\cP$.

\begin{nota}
  Define
$$\cP=\bigsqcup_{N=1}^{+\infty} \cP_N \times \N_0^{ \Z/ N\Z}.$$
\end{nota}}
%\laur{j'ai remplacé par $\N_0$ car je crois que $k_n \geq 0$?}

%\laur{j'ai retiré la notation $\partial P_n$ qui était peu utilisée et un peu confusionelle}
Suppose we have an integer $N\geq 1$, and an element $(P_n )_{n\in \Z/ N\Z}$ of $\cP_N$. % Write
% $(\partial P_n )_{n\in \Z/ N\Z}:=(g_{n-1}, g_n, h_n)_{n\in \Z/ N\Z}$.  
Let us consider a sequence of segments $I_n=[O_n, T_n]$ satisfying the prescriptions of the previous
section. The collection of compatible paths $(p_n)_{n\in \Z/ N\Z}$ is completely determined by the
data of the sequence $(g_{n-1}, g_n, h_n)_{n\in \Z/ N\Z}$,
$I_n=[O_n, T_n]$, {given that each $g_n$ comes with an orientation}. The data of a family of positive
integers $(k_n)_{n\in \Z/ N\Z}$ then defines a loop $\gamma$ via \eqref{e:desc}.
Note that if we just give {$(P_n )_{n\in \Z/ N\Z}$} and $(k_n)_{n\in \Z/ N\Z}$, there is some freedom
in the choice of $I_n=[O_n, T_n]$, but different choices will give loops $\gamma$ in the same
homotopy class.
 
 What we have just described is a map
\begin{align*}
\Phi : \cP \To \pi^*_1(S)
\end{align*}
onto the set of free homotopy classes of loops on $S$, which can then be composed by the projection
$\pi : \pi^*_1(S) \To {\mathrm{Top}}(S)$, the set of local topological types of loops in $S$.
 
If $\gamma$ is a non-simple loop, and if
$P=(P_n, k_n)_{n \in \Z \diagup N \Z}\in {\cP}$ is the sequence
associated to $\gamma$ via the construction of \S \ref{s:voluteconst}, then $\Phi(P)$ is homotopic
to $\gamma$. This shows that the map $\Phi$ is surjective. This is even true if we impose the
restriction $k_n=0$ if $\Sigma_n$ and $\Sigma_{n+1}$ are once-holed tori.

It it clear that an isotopy of the collection $(P_n)_{n\in  \Z/ N\Z}$ will result in an isotopy of the collection $(p_n)_{n\in \Z/ N\Z}$, and thus that the map $\Phi$ passes to the quotient of
$\cP$ by the {\emph{isotopy}} equivalence relation. However, it is much easier to count {\emph{homotopy}} classes of families $(P_n)_{n\in  \Z/ N\Z}$.
By Graaf and Schrijver \cite{graaf1997}, the isotopy equivalence relation is the same as the homotopy equivalence relation, modulo finite sequences of third Reidemeister moves. As a consequence, it remains true that the {homotopy} class of the collection $(P_n)_{n\in  \Z/ N\Z}$ determines the homotopy class of the collection $(p_n)_{n\in \Z/ N\Z}$. 

{
Hence, let us define $\tilde\cP$ to be the quotient of $\cP$ by the equivalence relation $  (P_n,
k_n)_{n\in  \Z/ N\Z} \sim (P'_n, k'_n)_{n\in \Z/N'\Z}$ if and only if
\begin{align}\label{e:req}
  \begin{cases}
    N=N' \\
    \exists j \in   \Z/ N\Z: \forall n, k_n= k'_{n+j} \\
    \exists \phi:S\To S \text{ homeomorphism} :
    \phi((P_n)_{n\in  \Z/ N\Z}) \text{ homotopic to }(P'_{n+j})_{n\in  \Z/ N\Z}
  \end{cases}
\end{align}}

We have shown that the map
\begin{align*}
\tilde\Phi : \tilde\cP \To {\mathrm{Top}}(S)
\end{align*}
is well defined and surjective. 

%\begin{nota}
%Later on, we may denote $\Phi^S, \cP^S, \tilde\Phi^S, \tilde\cP^S$ etc. if we wish to highlight the
%fact that the construction takes place in $S$. \laur{do we use this?}
%\end{nota}

\subsection{Filling and pairs of pants decomposition}

Before moving to the proof of Theorem \ref{t:TF_curves}, we clarify a fact that will help us count topological types. 

\begin{lem} \label{e:help_decompo} Let $S$ be a filling type of signature $(g_S, n_S)$ with
  $2g_S-2+n_S>0$. Let {$c=(c_i)_{1 \leq i \leq N}$ be a (possibly infinite)} collection of
  non-contractible loops in minimal position, which fills $S$. Then there exists a decomposition of
  $S= Q_1\cup Q_2\cup \ldots Q_{2g_S -2+ n_S} $ into surfaces of Euler characteristic $-1$ with
  disjoint interiors, and with the following property: if $Z$ is a bordered hyperbolic surface, if
  $\phi :S \To Z$ is a homeomorphism such that $\sum_{i=1}^N \ell_Z(\phi(c_i)) \leq L$, then
\begin{itemize}
\item $\ell(\partial Z)\leq 2L$;
\item $\ell_Z(\partial  (\phi (Q_m)))\leq 2L$;
\item $\ell_Z(\partial (\phi (S_m))) \leq 2L$ if $S_m:= Q_1\cup\ldots\cup Q_m$.
\end{itemize}
\end{lem}
Later on, we shall write $Q(c)=(Q_1(c), \ldots, Q_{2g_S -2+ n_S}(c))$ if we want to refer explicitly
to the decomposition into pairs of pants associated to $c$ via Lemma \ref{e:help_decompo}.
The multi-curve cutting $S$ into $Q_1\cup Q_2\cup \ldots Q_{2g_S -2+ n_S}$ will be denoted by $\lambda(c)=(\lambda_1(c),\ldots, \lambda_{3g_S -3+ n_S}(c))$.

\begin{rem}  \label{e:help_decompo3}(Variants of Lemma \ref{e:help_decompo}) If instead of assuming that  $\sum_{i=1}^N \ell_Z(\phi(c_i)) \leq L$ we just assume that $\ell_Z(\phi(c_i))\leq L$ for all $i$, the proof gives $\ell_Z(\partial  (\phi (Q_m)))\leq 2(m+1)L$ and $\ell_Z(\partial (\phi (S_m)))\leq 2(m+1)L$.

If the $c_i$, instead of being closed curves, are either closed curves with $\ell_Z(\phi(c_i))\leq L$ or pairs of pants with $\ell_Z^{\max}(\partial \phi(c_i))) \leq \tfL$, then an adaptation 
of the lemma gives 
$\ell_Z(\partial  (\phi (Q_m)))\leq (m+1) \max(2L, 3\tfL)$ and $\ell_Z(\partial (\phi (S_m)))\leq (m+1) \max(2L, 3\tfL)$.
\end{rem}

\begin{proof} The first item is standard, see \cite[Lemma 4.4]{Ours1}. 

%The surfaces $S_m$ and $Q_m$ are constructed inductively.
To simplify the discussion, we endow $S$ with an auxiliary Riemannian metric $Y$ and assume that the $c_i$ are geodesics. We construct $S_m$ and $Q_m$ inductively, they are defined to have geodesic boundary for the auxiliary metric $Y$.

In our construction, $S_m$ will be filled by a union $p_1\cup\ldots \cup p_m\cup p_{m+1}$, where each $p_j$ is a closed subsegment of $c$, joining two intersection points of $c$. The $p_j$ only intersect transversally; the origin of $p_{m+1}$ lies on $p_1\cup\ldots \cup p_m$, and its endpoint lies either in $p_1\cup\ldots \cup p_m$, or in $\overset{\circ}{p}_{m+1}$ (the segment $p_{m+1}$ deprived of its boundary points).

$Q_1$ is constructed as follows: if one of the loops $c_i$ is non-simple, we take $p_1$ to be a volute of $c_i$ and $p_2=\emptyset$. Then $Q_1=S_1$ is the surface filled by $p_1$.
If all the loops $c_i$ are simple, then there must exists $i<j$ such that $c_i\cap c_j\not=\emptyset$. We take $p_1=c_i$ and $p_2$ a sub-arc arc of $c_j$ beginning and ending on $c_i$. Then $Q_1=S_1$ is the surface filled by $p_1\cup p_2$.

Suppose we have constructed $p_1,\ldots , p_m$, $Q_1, \ldots, Q_{m-1}$, with $S_{m-1}=Q_1\cup\ldots\cup Q_{m-1}$ filled by $p_1\cup\ldots \cup p_m$.

If $S_{m-1}=S$, there is nothing to be done. Otherwise, there is a (geodesic) subarc $p$ of $c$ leaving from $\partial S_{m-1}$ and returning to it, contained in $S\setminus S_{m-1}$.

The surface
filled by $p_1\cup\ldots \cup p_m\cup p$ is strictly larger than $S_{m-1}$. If $p$ is simple, it intersects either one or two boundary components of $\partial S_{m-1}$. Then $p\cap (S\setminus S_{m-1})$,
together with those boundary components, fill a surface $Q_{m+1}$, of Euler characteristic $-1$. The full geodesic containing $p$ must enter and exit $S_{m-1}$. Hence, it must intersect the already constructed $p_1\cup\ldots \cup p_m$ (otherwise, $p_1\cup\ldots \cup p_m$ would not fill $S_{m-1}$). Thus, there is a geodesic arc $p_{m+1}$ containing $p$, ending and beginning at $p_1\cup\ldots \cup p_m$, such that $p_{m+1}\setminus p$ is fully contained in $S_{m-1}$.

If $p$ is not simple, then parametrize $p$ by $p: [0, T]\To S\setminus S_{m-1}$, such that $p(0)\in \partial S_{m-1}$. Take the smallest $t\leq T$ such that $p(t)\in p([0, t))$. We define $Q_{m+1}$ to be the surface filled by $p([0, T])$ and the component of $\partial S_{m-1}$ containing $p(0)$. We extend $p$ to $p_{m+1}$ as previously.

Now, let $Z$ be a bordered hyperbolic surface, and let $\phi :S \To Z$ is a homeomorphism. If $\gamma=(\gamma_1, \ldots, \gamma_N)$ are the geodesic representatives of $(\phi(c_1), \ldots, \phi(c_{N}))$ for the metric $Z$, then $(\gamma_1, \ldots, \gamma_N)$ are obtained from $(\phi(c_1), \ldots, \phi(c_{N}))$ by an isotopy followed by a finite sequence of third Reidemeister moves (Graaf and Schrijver \cite{graaf1997}). This operation maps each subarc $\phi(p_j)$ to a subarc $q_j$ of $\gamma$, such that $q_1\cup\ldots \cup q_{m}$ fills a surface isotopic to $\phi(S_{m-1})=\phi(Q_1)\cup\ldots\cup \phi(Q_{m-1})$. By construction, for $i\not=j$, $q_i\cap q_j$ consists of isolated points, so $\ell(q_1\cup\ldots \cup q_{m})=\sum_{j=1}^m \ell(q_j) \leq \ell(\gamma)$.

The assumption $\ell_Z(c)\leq L$ means that $\ell(\gamma)\leq L$, so that $\ell(q_1\cup\ldots \cup q_{m})\leq \ell(\gamma)\leq L$. By \cite[Lemma 4.4]{Ours1}, we can say that $\ell_Z(\partial (\phi( S_{m-1})))\leq 2L$. Even more precisely,
 $\ell_Z(\partial (\phi( S_{m-1})))\leq 2\ell(q_1\cup\ldots \cup q_{m})$ and $\ell(\partial Q_{m+1})\leq \ell_Z(\partial (\phi (S_m))) + 2\ell(q_{m+1})\leq  2\ell(q_1\cup\ldots \cup q_{m+1}) \leq 2L$. This ends the proof.
\end{proof}

\subsection{Proof of Theorem \ref{t:TF_curves}}We now describe a subset of $\tilde\cP$ which projects onto $\mathrm{Loc}_{S}^{\inj,\tfL, L}$, and evaluate its cardinality. Because of the structure of the equivalence relation \eqref{e:req}, we can write
$\tilde\cP=\bigsqcup_{N=1}^{+\infty}\tilde \cP_N$, where $\tilde \cP_N$ is the set of equivalence classes restricted to  $\cP_N \times \Z^{ \Z/ N\Z}$.

 Our construction starts with the following three observations.
 
  \begin{lem}\label{l:int1}
    If $S$ is endowed with a $(\inj, \tfL)$ tangle-free hyperbolic metric $Z$, and if $\gamma$ is a
    loop in $S$ such that $\ell_Z(\gamma)\leq L$, then
    $$i(\gamma, \gamma)\leq \frac{4}{\inj^2} L^2.$$
\end{lem}
See for instance \cite[Expos\'e 4, Lemme 2]{FLPf}; the constant $C$ given in this proof is smaller than $\frac{4}{\inj^2}$, where $\inj$ is a lower bound on the injectivity radius.

  \begin{lem}\label{l:int2}
    Supposed $S$ is endowed with a $(\inj, \tfL)$ tangle-free hyperbolic metric $Z$, and $\gamma$ is
    a loop filling $S$ such that $\ell_Z(\gamma)\leq L$. {Let}
    $\lambda=(\lambda_1(\gamma), \ldots, \lambda_{3g_S-3 +n_S}(\gamma))$ be the multicurve cutting
    $S$ into surfaces of Euler characteristic $-1$, defined by Lemma \ref{e:help_decompo}. Then
$$i(\gamma, \lambda_j)\leq 8 \frac{L^2}{\inj^2} ,$$
for all $j \in \{1, \ldots, 3g_S-3 +n_S\}$. More generally, for any multicurve $c$,
$$i(c, \lambda_j)\leq 8 \frac{L \ell_Z(c)}{\inj^2}.$$
\end{lem}
This follows again from \cite[Expos\'e 4, Lemme 2]{FLPf}, and the lower bound $\inj$ on the injectivity radius.
 
 \begin{lem} \label{l:bound_N_k}
Suppose $S$ is endowed with a $(\inj, \tfL)$ tangle-free hyperbolic metric $Z$, and $\gamma$ is a loop on $S$ such that $\ell_Z(\gamma)\leq L$. Let $P=(P_n, k_n)_{n\in  \Z/ N\Z} \in \cP_N\times \N^{\Z/N\Z}$ be the sequence associated to $\gamma$ via the construction of \S \ref{s:voluteconst}. Then
$$ N\leq 6 \frac L{\tfL}$$
and each $k_n$ satisfies 
$$ 0\leq k_n\leq 1+\frac{L}{\inj}.$$
\end{lem}
\begin{proof}
%Write $\partial P_n=(g_{n-1}, g_n, h_n)$.
  The tangle-free hypothesis implies that, for each $n \in \{1, \ldots, N\}$,
  $$\max(\ell_Z(g_n),\ell_Z(g_{n-1}),\ell_Z(h_n))\geq \tfL.$$
  As a consequence of \eqref{e:boundGn}, the volutes $G(x_0), \ldots, G(x_{N-1})$ have lengths at
  least $\tfL / 2$, and by \eqref{e:bound_volutes},
$$\ell_Z(\gamma)\geq  N\frac{\tfL}6,$$
which is the announced result.

To obtain the upper bound on $k_n$, we recall that $\gamma$ contains at least $k_n-1$ copies of
closed curves homotopic to $g_n$. By the minimizing property of closed geodesics, this implies
$$\ell_Z(\gamma)\geq (k_n-1) \ell_Z(g_n).$$
By the tangle-free condition, $\ell_Z(g_n)\geq \inj$, which gives the desired bound.

 \end{proof}

{We are now ready to prove Theorem \ref{t:TF_curves}.}  

\begin{proof}
  We now use Lemmas \ref{l:int1} and \ref{l:int2} to describe finite sets $\tilde \cP_N^{\inj,\tfL,
    L}\subset \tilde\cP_N$ such that if
  $$ \tilde \cP^{\inj,\tfL, L}
    := \bigsqcup_{N=1}^{{\lfloor 6 \frac L{\tfL} \rfloor}}
    \tilde \cP_N^{\inj,\tfL, L} \times \Big[0,1+\frac{L}{\inj} \Big]^{\Z/N\Z} $$
   then the map
  \begin{align}\label{e:Phi_bound}
    \tilde\Phi : \tilde \cP^{\inj,\tfL, L} \To  \mathrm{Loc}_{S}^{\inj,\tfL, L}
  \end{align}
  is surjective. Theorem \ref{t:TF_curves} will follow from an estimate of
  $\#\tilde \cP_N^{\inj,\tfL, L} .$

  Let $\bar\Lambda$ be the set of decompositions of $S$ into surfaces of Euler characteristic $-1$
  up to isotopy.
  Let us fix a finite family $\Lambda_S\subset \bar\Lambda$, containing exactly one representative
  of each orbit under the group of homeomorphisms on $\bar\Lambda$. The cardinality of $\Lambda_S$
  is a function of $(g_S, n_S)$, that we do not need to make explicit. For a given decomposition
  $\lambda=(\lambda_1, \ldots, \lambda_{3g_S-3+n_S})\in\Lambda$, we denote by
  $\tilde\cP_{N}^{\inj, \tfL, L, \lambda}$ the subset of $\tilde \cP_N$, formed of elements
  $(\bar P_1, \ldots, \bar P_N)$ satisfying:
  $i(\partial P_j,\partial P_k) \leq \frac{4}{\inj^2} (2L)^2$ for all $j\not= k$ and
  $i(\partial P_j, \lambda_k)\leq \frac{4}{\inj^2} (2L)^2$ for all $j, k$.  Next, define
  $ \tilde \cP_N^{\inj,\tfL, L}=\bigcup_{\lambda\in \Lambda_S}\tilde\cP_{N}^{\inj, \tfL, L, \lambda}.$

  If $S$ is endowed with a tangle-free metric $Z$ such that $\ell(\gamma)\leq L$, by
  \eqref{e:boundGn}, $\ell_Z(\partial P_j)\leq 2L$.  Lemmas \ref{l:int1}, \ref{l:int2} and
  \ref{l:bound_N_k} imply that \eqref{e:Phi_bound} is surjective.

  Proposition \ref{p:fundamental_count} with $I=J= \frac{4}{\inj^2} (2L)^2= \frac{16}{\inj^2} L^2$
  implies that
  \begin{align*}
    \# \tilde \cP_N^{\inj,\tfL, L}\leq 
    2^{ 3g_S-3 +n_S} \# \Lambda_S\,\,  \Big(3(16\frac{L^2}{\inj^2} +1)\Big)^{2(3g_S-3   +n_S)N } 
  \end{align*}
 
  Finally, summing over all elements of $\Big[0,1+\frac{L}{\inj} \Big]^{\Z/N\Z}$ and over all
  possible $N\leq 3 \frac L{\tfL}$,
  \begin{align*}
    \# \mathrm{Loc}_{S}^{\inj,\tfL, L} 
    & \leq   \#\tilde \cP^{\inj,\tfL, L}\\
    & \leq 3 \frac L{\tfL} 2^{ 3g_S-3 +n_S} \# \Lambda_S\,\, 
      \Big(3(\frac{16}{\inj^2} L^2+1)\Big)^{2(3g_S-3   +n_S) }  \Big)^{{\lfloor 6 \frac L{\tfL} \rfloor}}
      \times \Big( 1+\frac{L}{\inj}\Big)^{{\lfloor 6 \frac L{\tfL} \rfloor}},
  \end{align*}
  which proves Theorem \ref{t:TF_curves}.
\end{proof}
 \appendix
 \section{More counting!}
 \label{s:lcsurface}

 Here we include an additional counting result whose proof uses similar arguments as for Theorem \ref{t:TF_curves}, and which is needed in our papers \cite{Expo, Ours2}.
 The objects that we need to count are now pairs formed by a loop $\gamma$ and a tangle $\tau$.
 
 \subsection{Local topological types of lc-surfaces.}
 The topological object to consider is now a pair formed by $(\gamma, \tau)$, where $\gamma$ is a loop and $\tau$
  is a c-surface. Such a pair will be called an lc-surface. We shall define topological types of such objects, following the ideas of \cite[\S 4]{Ours1}.

   \begin{nota}\label{n:FTlc} To any $q\geq 0$ and any $(\vg, \vn)\in {\mathfrak{S}}^{q}$
   we shall associate a \emph{fixed} smooth
  oriented c-surface $S$ of signature $(\vg, \vn)$. The data of $(\vg, \vn)$, or equivalently of the c-surface $S$, is called a \emph{c-filling
    type}.
\end{nota}

\begin{defa}
  \label{def:local_type_lc}
  A \emph{local lc-surface} is a triplet $(S,\gamma, \tau)$, where $S$ is a c-filling type,
  $\gamma$ is a loop in $S$, $\tau$ is a sub-c-surface of $S$, and the pair $(\gamma, \tau)$ fills $S$. Two local lc-surfaces $(S,\gamma, \tau)$ and
  $(S',\gamma', \tau')$ are said to be \emph{locally equivalent} if $S=S'$ (i.e.
  $g_S = g_{S'}$ and $n_S=n_{S'}$), and there exists a positive homeomorphism
  $\psi : S \rightarrow S$, possibly permuting the boundary components of $S$ and the connected components of $S$,
  such that $\psi \circ \gamma$ is freely homotopic to $\gamma'$, and $\psi(\tau)=\tau'$. The latter equality means that for $\tau=(\tau_1, \ldots, \tau_q), \tau'=(\tau'_1, \ldots, \tau'_{q'})$, we have $q=q'$, $\psi(\tau_j)=\tau'_j$, $\psi$ respects the orientation of $\tau_q$ if $\tau_q$ is 1d, or the numbering of the boundary components of $\tau_q$ if $\tau_q$ is 2d. This defines
  an equivalence relation $\eq$ on local lc-surfaces. Equivalence classes for this
  relation are denoted as $\eqc{S, \gamma, \tau}$ and called \emph{local (topological)
    types} of lc-surfaces.

\end{defa}

 \begin{defa} \label{e:def_Tan} 

Denote by $\mathrm{Tan}_{S,\chi, M}^{\inj, \tfL, L}$ the set of local topological types $\eqc{S, \gamma, \tau}$ of lc-surfaces of filling type $S$, such that $\eqc{S, \gamma}=T\in \mathrm{Loc}_{\chi}^{\inj,\tfL, L}$,
$\chi(\tau)<M$, and:
there
exists a hyperbolic surface $Z$, a homeomorphism $\phi : S\To Z$, with
  $\ell_Z(\phi(\gamma))\leq L$, and such that the geodesic representative of $\phi(\tau)$ is a $(\inj, \tfL)$-derived tangle in $Z$.
\end{defa}

 \begin{thm}   \label{p:count_with_tangles}  Assume that $S$ is purely 2-dimensional.Then
$$\# \mathrm{Tan}_{S,\chi, M}^{\inj, \tfL, L} \leq \O[ \chi, S]{   \frac L{\tfL}
    \Big(3(16 \frac{L^2}{\inj^2} +1)\Big)^{9\chi   \frac L{\tfL}}   \times \Big( 1+\frac{L}{\inj}\Big)^{3 \frac L{\tfL}}  \Big(2L e^{3M\tfL}\Big)^{3\chi(S)}}.$$
\end{thm}

Recall that $\chi(S)$ is the opposite of the Euler characteristic of $S$, in particular it is positive.

As a corollary,
\begin{cor} \label{p:count_with_tangles_bis} If $L=A\log g$, $\tfL=\alpha \log g$, and if $\chi(S)< \chi'$,
$$\# \mathrm{Tan}_{S,\chi, M}^{\inj, \tfL, L} \leq \O[ \chi', \inj, \alpha, A]{ 
 g^{10 M\chi'\inj}   }.$$
\end{cor}
 
Note that we are no longer invoking the absence of tangles in Definition \ref{e:def_Tan} and Theorem \ref{p:count_with_tangles}. As a consequence, we cannot use lower bounds on the injectivity radius to find upper bounds on the intersection numbers. The upper bound is still polynomial in $L$, but has an exponential factor in $\tfL$.

Instead of the tangle-free assumption, we use the following fact in hyperbolic trigonometry to restrict the intersection numbers:
 \begin{lem}\label{l:lower_bound_in_tangle}
 There exists a universal constant $C$ such that, for any $\tfL>0$, for any hyperbolic surface $Z$ of Euler characteristic $-1$ having $\ell^{\max}(\partial Z)\leq \tfL$, any segment $c$ going from the boundary $\partial Z$ to itself, and not homotopic to a portion of $\partial Z$, has length
 $$\ell(c)\geq C e^{-\tfL}.$$
 \end{lem}
 
 \begin{cor} Let $Y$ be a $(\inj, \tfL)$-derived tangle of signature $(g, n)$, with
   $0<2g-2+n<M$. Consider a maximal multi-curve $\lambda=(\lambda_1, \ldots, \gamma_{3g-3 + 2n })$,
   cutting $Y$ into surfaces of Euler characteristic $-1$, such that $\ell(\lambda_t)\leq 3\tfL M$
   for all $t \in \{1, \ldots, 3g-3 + 2n\}$ (the existence of which was proven in Lemma
   \ref{e:help_decompo}).
 If $\gamma$ is a geodesic in $Y$ of length $\ell(\gamma)\leq L$, then % $(i(\gamma, \lambda)-1)C e^{-3\tfL M}\leq L$. In other words,
 $$ i(\gamma, \lambda)\leq 1 +\frac1{C} e^{3\tfL M}L.$$
 
 \end{cor}
 \begin{proof} If $i(\gamma, \lambda)=n$, we can find $n-1$ disjoint segments on $\gamma$, joining components of $\lambda$, of length $\geq C e^{-3\tfL M}$ by Lemma \ref{l:lower_bound_in_tangle}.
 \end{proof}

 The proof of Theorem \ref{p:count_with_tangles} is based on the following facts: with the notation
 of Definition \ref{e:def_Tan}, let in addition $S'$ be the surface filled by $\gamma$. Denote as
 previously
 $\lambda=(\lambda_1, \ldots, \lambda_{3g_{S'}-3+n_{S'}})=(\lambda_1(\gamma), \ldots,
 \lambda_{3g_{S'}-3+n_{S'}}(\gamma))$ be the multicurve associated to $\gamma$ via Lemma
 \ref{e:help_decompo}. Denote also
 $\lambda_{3g_{S'}-3+n_{S'}+1}, \ldots,\lambda_{3g_{S'}-3+2n_{S'}}$ the boundary components of
 $S'$.

 Let $P=(P_1, \ldots, P_N, k_1, \ldots, k_N)\in \cP^{S'}_N\times \N^{\Z/ N \Z}$ be the sequence
 associated to $\gamma$ via the construction of \S \ref{s:voluteconst}. Then, in addition to the
 restrictions on $P$ that we found before, Corollary~\ref{l:lower_bound_in_tangle} tells us that we
 must have, for all $j \in \{1, \ldots ,3g_{S'}-3+2n_{S'}\}$,
$$i(\lambda_j, \partial \tau)\leq \frac1{C} 2L e^{3M\tfL}$$
and, for all $j \in \{1, \ldots, N\}$,
$$i(P_j, \partial \tau)\leq \frac1{C} 2L e^{3M\tfL}.$$
The arguments are then similar to those for Theorem \ref{t:TF_curves} and we omit the details.

\bibliographystyle{plain}
\bibliography{bibliography}

 \end{document}